\documentclass[12pt]{amsart}

\usepackage[T1]{fontenc}
\usepackage[english]{babel}

\usepackage{fourier}

\usepackage{amssymb,latexsym, amscd, a4wide}
\usepackage[all]{xy}
\usepackage{graphicx}
\usepackage{mathrsfs}
\usepackage{amsmath}
\usepackage{pb-diagram}
\usepackage{bbold}
\usepackage{color}
\usepackage[dvipsnames]{xcolor}

\usepackage{setspace}

\newtheorem{thm}{Theorem}[section]
\newtheorem{lem}[thm]{Lemma}
\newtheorem{prop}[thm]{Proposition}
\newtheorem{cor}[thm]{Corollary}

\newtheorem{defi}[thm]{Definition}
\newtheorem{rem}[thm]{Remark}

\newtheorem*{thm*}{Theorem}
\newtheorem*{prop*}{Proposition}

\numberwithin{equation}{section}
\newcommand{\nc}{\newcommand}
\nc{\ssn}{\subsection{}} \nc{\sssn}{\subsubsection{}}

\newcommand{\dem }{\noindent\textbf{Proof}. }
\newcommand{\findem }{\hfill $\Box$ \vskip0.5cm }

\newcommand{\Z}{\ensuremath{\mathbb{Z}}}

\newcommand{\C}{\ensuremath{\mathbb{C}}}
\newcommand{\A}{\ensuremath{\mathbb{A}}}

\newcommand{\g}{\ensuremath{\mathfrak{g}}}
\newcommand{\h}{\ensuremath{\mathfrak{h}}}

\nc{\id}{\textrm{id}}

\newcommand{\Hom}{\operatorname{Hom}}

\newcommand{\rank}{\operatorname{rank}}

\newcommand{\Sym}{\operatorname{Sym}}

\newcommand{\mc}[1]{\mathcal{#1}}

\newcommand{\mb}[1]{\mathbb{#1}} 

\begin{document}

\title{Representations of twisted quantum  affine algebras}

\author{Juan Camilo Arias}
\author{Vyacheslav Futorny}
\author{Kailash C. Misra}

\address{School of Sciences and Engineering, Universidad del Rosario, Bogot\'a, Colombia}
\email{juancamil.arias@urosario.edu.co }
\address{Shenzhen International Center for Mathematics, Southern University of Science and Technology, China}
\email{vfutorny@gmail.com}
\address{Department of Mathematics, 
 North Carolina State University, 
 Raleigh, NC, USA}
\email{misra@ncsu.edu}


\maketitle

\begin{abstract}
We develop the representation theory of imaginary Verma modules for twisted quantum affine algebras and construct the corresponding Kashiwara algebras. The twisted case presents substantial new difficulties compared with the untwisted setting: the PBW root vectors have nontrivial orbit structure, the imaginary root spaces occur with multiplicities, and roots of unity enter essentially into the defining commutation relations.

Our first main result is an explicit PBW-type basis for twisted quantum imaginary Verma modules associated with the natural imaginary partition of the affine root system. We reorganize the PBW basis of the twisted quantum affine algebra in a form adapted to this partition and use it to construct integral forms of the imaginary Verma modules. We prove that their classical limits are the corresponding imaginary Verma modules for twisted affine Lie algebras. This provides, in particular, a precise compatibility between the twisted quantum and classical theories that is not immediate from the standard PBW theory.

We then determine the structure and irreducibility of the twisted quantum imaginary Verma modules. We prove that the Heisenberg submodule generated by the imaginary root vectors is irreducible precisely at nonzero central charge, and establish the corresponding irreducibility criterion for the reduced twisted imaginary Verma modules at zero central charge.

The second main part of the paper introduces the Kashiwara algebra in the twisted setting. A crucial ingredient is an explicit current realization of the twisted quantum affine algebra. The required current formula for the generators $x_i^\pm(u)$ differs essentially from the untwisted formula and is needed to construct the Omega operators. We derive the resulting Omega-operator commutation relations, including the root-of-unity factors specific to the twisted cases. These relations lead to a new presentation of the Kashiwara algebra associated with the reduced twisted imaginary Verma modules. We prove that the negative current algebra is a simple module over this Kashiwara algebra and construct a symmetric non-degenerate bilinear form characterized by the Omega operators. Thus, the paper extends the Kashiwara-algebra approach from untwisted to twisted quantum affine algebras while revealing new algebraic phenomena caused by the twisting.
\end{abstract}

\section{Introduction}

Quantum affine algebras and their representations play a central role in the theory of quantum groups, with important connections to affine Lie algebras, integrable systems, vertex representations, and canonical and crystal bases. The representation theory of quantum affine algebras is particularly rich because, in addition to the standard highest weight theory associated with the usual triangular decomposition, there are other natural choices of positive and negative systems leading to different classes of representations. One important example is provided by imaginary Verma modules, introduced for affine Lie algebras in \cite{Fut}. These modules are obtained from a non-standard partition of the affine root system in which all real roots of the form $\alpha+k\delta$, with $\alpha$ a positive root of the underlying finite-dimensional Lie algebra and $k\in\mathbb Z$, are placed on the same side, while the positive imaginary roots are placed on the opposite side. This construction leads to highest weight modules with a structure substantially different from that of the usual Verma modules.

The quantum analogues of imaginary Verma modules were studied for untwisted quantum affine algebras in \cite{CFM01,AFM}.  When the central charge is nonzero, this Heisenberg subalgebra acts irreducibly on the corresponding finite weight part, and consequently the imaginary Verma module itself is irreducible. At zero central charge, a nontrivial maximal submodule appears naturally in the imaginary direction. The resulting quotient, usually called the reduced imaginary Verma module, has a particularly interesting representation-theoretic structure and is closely related to Kashiwara algebras and crystal bases.

The purpose of the present paper is to extend this theory to \emph{twisted quantum affine algebras}. Although many of the structural ideas are analogous to the untwisted case, the twisted setting introduces substantial additional difficulties. In particular, the imaginary root spaces have nontrivial multiplicities and depend on the twisting automorphism, the corresponding root vectors involve the orbit structure of the Dynkin diagram, and the commutation relations contain factors determined by roots of unity. These features require a separate treatment of the PBW bases, the classical limits, and the operators governing the representation theory.

Our first objective is to develop a form of the PBW theory for twisted quantum affine algebras that is adapted to the non-standard partition defining imaginary Verma modules. We use the convex ordering and root vectors constructed by Damiani \cite{D}, but reorganize the resulting PBW monomials according to the decomposition of the positive and negative roots into real, imaginary, and mixed parts. This ordering makes it possible to control the action of the negative part on the highest weight vector and to obtain an explicit spanning set, and subsequently a basis, for twisted quantum imaginary Verma modules.

The corresponding integral form over
\[
\mathbb A=\mathbb C[q,q^{-1},[k]_{q_i}^{-1}\mid i\in I_0,\ k>0]
\]
is then constructed, and its specialization at $q=1$ is identified with the imaginary Verma module for the corresponding twisted affine Lie algebra
\[
M_q(\lambda)=U_q(\widehat{\mathfrak g}^{(r)})
\otimes_{B_q}\mathbb C(q)v_\lambda.
\]
The PBW ordering allows us to describe the vectors of this module by ordered monomials involving the negative real root vectors and the negative imaginary root vectors. This description is essential for all subsequent arguments, since it provides a concrete filtration by height and allows us to control lower-order terms arising from the twisted commutation relations.

We next construct an $\mathbb A$-form of $M_q(\lambda)$ and study its classical limit. The resulting specialization is shown to recover the imaginary Verma module for the twisted affine Lie algebra $\widehat{\mathfrak g}^{(r)}$. Thus the quantum construction is compatible with the corresponding classical representation theory. This compatibility is also used in the analysis of reducibility. In particular, the Heisenberg submodule generated by the imaginary root vectors is irreducible precisely when the central charge is nonzero, and consequently the twisted quantum imaginary Verma module is irreducible exactly at nonzero level. At level zero, we obtain the reduced twisted imaginary Verma module by factoring out the submodule generated by the nontrivial proper part of the Heisenberg module. Under the corresponding nonvanishing conditions on the finite-dimensional Cartan directions, this reduced module is irreducible.

The second part of the paper concerns the construction of Kashiwara algebras associated with these reduced twisted imaginary Verma modules. For untwisted quantum affine algebras, Kashiwara algebras provide an efficient algebraic framework for studying crystal bases and the structure of the negative part of the quantum group \cite{CFM01,AFM}. The construction in the twisted case requires a careful analysis of the Drinfeld current realization. In particular, the current commutation relations involve the twisting automorphism $\sigma$ and an $r$-th root of unity $\omega$. The resulting formulas are substantially more involved than in the untwisted situation. 

We introduce the operators
\[
\sum_{k\in\mathbb Z}\Omega_{\phi_i}(k)u^{-k},
\]
acting on the negative current algebra $\mathcal N_q^-$. These operators play the role of the Omega operators in the untwisted theory, but their defining relations contain the additional orbit and root-of-unity factors characteristic of the twisted setting. We establish their commutation relations with the negative current generators and with one another. These relations provide the fundamental algebraic input for the construction of the twisted Kashiwara algebra $\mathcal K_q$, for which $\mathcal N_q^-$ is a simple left module.

The main results of the paper can therefore be summarized as follows.

\begin{itemize}
\item We construct an ordering of the twisted PBW root vectors adapted to the natural imaginary partition and use it to obtain explicit bases for twisted quantum imaginary Verma modules.

\item We construct integral forms of these modules and identify their classical limits with imaginary Verma modules for twisted affine Lie algebras.

\item We determine the irreducibility criterion for twisted quantum imaginary Verma modules and for their reduced quotients.

\item Using the Drinfeld realization, we construct the Omega operators associated with twisted quantum affine algebras and establish their defining commutation relations.

\item We construct the Kashiwara algebra associated with the reduced twisted imaginary Verma modules and prove that the negative current algebra is a simple module over this algebra.

\item We construct a non-degenerate symmetric form on the negative current algebra, characterized by the action of the Omega operators.
\end{itemize}

The twisted cases $r=2$ and $r=3$ exhibit genuinely new phenomena that are absent in the untwisted theory. The orbit structure of the Dynkin diagram and the associated roots of unity enter directly into the current relations and consequently into the defining relations of the Kashiwara algebra. Thus, although the overall strategy is inspired by the untwisted theory, the resulting algebras and their representations require new calculations and a separate treatment.

The paper is organized as follows. Section 2 recalls the root data, twisted quantum affine algebras, their root vectors, commutation relations, and the PBW basis in a form suitable for the subsequent arguments. Section 3 introduces twisted quantum imaginary Verma modules and establishes their PBW-type description. Section 4 constructs the corresponding $\mathbb A$-forms and studies their classical limits. Section 5 investigates the structure and irreducibility of twisted quantum imaginary Verma modules and their reduced quotients. Section 6 develops the Drinfeld realization and constructs the Omega operators together with their commutation relations. Finally, Section 7 introduces the associated Kashiwara algebra, proves the simplicity of the negative current algebra as a Kashiwara module, and constructs the corresponding non-degenerate bilinear form.

\section{Preliminaries}\label{prelims}

Let $\g$ be a simple finite dimensional Lie algebra over $\C$ of simply laced type $X_N$. Let $\sigma$ be an automorphism of $\g$ of order $r= 1,2,3$ and denote by $\hat{\g}^{(r)}$ the associated affine Lie algebra of type $X_N^{(r)}$. When $r=1$, it is the untwisted affine Lie algebra associated with $\g$; in other cases, it is the twisted affine Lie algebra that can be realized as

$$ \hat{\g}^{(r)}  :=  \Big(\bigoplus_{m\in \mb{Z}} \g^{[m \mod r]} \otimes \mb{C}t^{m} \Big) \oplus \C c \oplus \C d$$

where $c$ is a central element, $d$ a derivation, and $\g^{[m \mod r]} = \{ x \in \g \ |\ \sigma(x) = \omega^{m}x\}$. It is known, in the twisted case, that the possibilities for $X_N^{(r)}$ are the types $A_2^{(2)}$, $A_{2n}^{(2)}, \ (n\geq 2)$, $A_{2n-1}^{(2)}, \ (n \geq 3)$, $D_{n+1}^{(2)}, \ (n \geq 2)$, $E_6^{(2)}$, and $D_4^{(3)}$, where $N = 2, 2n, 2n-1, n+1, 6, 4$, respectively. Recall that $n$ is the number of vertices in the Dynkin diagram of $\g^{[0]} = \{x \in \g \ | \ \sigma(x) = 0\}$, which is a simple Lie algebra of types $A_1$, $B_n$, $C_n$, $B_n$, $F_4$, and $G_2$, respectively. \\

Set the index sets $I=\{0, 1, \dots, N\}$, $I_0 = I \setminus \{0\}$, and $I_\sigma = \{1, 2, \dots, n\}$. Let $A=(a_{ij})_{i,j\in I_r}$ denote the affine Cartan matrix for the twisted affine Lie algebra $\hat{\g}^{(r)}$, where $I_r=I$ when $r=1$ and $I_r=I_\sigma\cup\{0\}$ when $r=2,3$. Let $D = \text{diag}(d_0, d_1, \dots, d_N)$ be the diagonal matrix containing relatively prime positive integers $d_i$ such that $DA$ is symmetric. Let $\hat{\h} = \text{span}\{h_i \ |\ i\in I_r \}\oplus \C d$ be the  Cartan subalgebra of $\hat{\g}^{(r)}$, and the symmetric bilinear form satisfies $(\alpha_i \mid \alpha_j) = d_i a_{ij}$ along with $(\delta \mid \alpha_i) = (\delta \mid \delta) = 0$ for all $i,j \in I_r$. The form $( \ | \ ) / r$ induces a nondegenerate invariant bilinear form on $\hat{\g}^{(r)}$ ($r=2,3$) denoted by $\langle | \rangle$. Following the convention of Kac, the scaling factors $d_i$ for the diagonal matrix $D$ are explicitly given by:
$$ d_i = \begin{cases} 
1 & \text{if } X_N^{(r)} = A_{2n}^{(2)} \text{ and } i \neq 0, \\
r & \text{if } X_N^{(r)} \neq A_{2n}^{(2)} \text{ and } i \in \sigma\text{-orbit of } 0, \\
1 & \text{otherwise.}
\end{cases} $$

\subsection{Root datum} Let $\Delta$  be the root system of $\hat{\g}^{(1)}$ with simple roots $\Pi=\{\alpha_0', \alpha_1', \ldots, \alpha_n'\}$, and let $\delta=\alpha_0'+\theta$ be the null root where $\theta$ is the longest root of the underlying simple Lie algebra $\g$. Recall that $\Delta = \Delta^{re} \cup \Delta^{im}$, where $\Delta^{re}$ and $\Delta^{im}$ denote the real and imaginary sets of roots. Let $Q$, $P$, $Q^{\vee}$, and $P^{\vee}$ denote the root lattice, the weight lattice, the coroot lattice, and the coweight lattice respectively.\\

By definition $P^{\vee} := \bigoplus_{i\in I}\Z \omega_i^{\vee}$, where $\omega_i^{\vee}$ are the fundamental coweights defined by the equation $\langle \omega_i^{\vee} , \alpha_j \rangle = \delta_{ij}$. The weight lattice $P$ is the sub-lattice of $P^{\vee}$ defined by $P := \bigoplus_{i\in I}\Z \omega_i$, where $\omega_i = d_i\omega_i^{\vee}$.\\

The root system with multiplicities is denoted by $\tilde{\Delta}$. It is by definition $\tilde{\Delta} = \Delta^{re} \cup \tilde{\Delta}^{im}$ where $\tilde{\Delta}^{im} = \{(m\delta, i) \ | \ m\in \Z\setminus\{0\},\  \tilde{d}_i|m \}$ and $\tilde{d_i} = 1$ in the untwisted case or in the case $A_{2n}^{(2)}$, and $\tilde{d}_i = d_i$ for the other cases. We denote by $p: \tilde{\Delta} \to \Delta$ the natural projection mapping $(m\delta, i) \mapsto m\delta$. Let $\Delta_0$ be the finite root system of $\g^{[0]}$ and $\Pi_0=\{\alpha_1, \dots, \alpha_n\}$ be its corresponding simple roots.\\

Recall that a subset $S \subset \Delta$ is said to be closed if for any $\alpha, \beta \in S$ such that $\alpha + \beta \in \Delta$, we have $\alpha + \beta \in S$. A subset $S$ forms a closed partition of $\Delta$ if $S$ is closed, $\Delta = S \cup (-S)$, and $S \cap (-S) = \emptyset$. In this paper, we consider a non-standard closed partition of $\Delta$ (inequivalent to the standard partition $\Delta_+$), given by:
$$ S = \{ \alpha + k\delta \ \vert \ \alpha\in \Delta_{0,+}, \ k\in \Z \} \cup \{ k\delta \ \vert \ k>0 \} $$
satisfying $\Delta = S \cup (-S)$. This partition forms the precise foundational framework for constructing our imaginary Verma modules.

\subsection{Twisted quantum affine algebras}

Let $\hat{\g}^{(r)}$ be a twisted affine Lie algebra of type $X_N^{(r)}$, and let $U_q(\hat{\g}^{(r)})$ be its quantum affine algebra (see \cite{CP}, \cite{L01}), i.e., the associative and unital $\C(q^{1/2})$-algebra with generators $E_i, F_i, K_{\alpha}, \gamma^{\pm1/2}, D^{\pm 1}$ for $0\leq i \leq n$, $\alpha\in Q$, and defining relations:

$$ DD^{-1} = D^{-1}D = K_\alpha K_{-\alpha} = K_{-\alpha} K_\alpha = \gamma^{1/2}\gamma^{-1/2}=1 $$
$$ [\gamma^{\pm1/2}, U_q(\hat{\g}^{r})]=0$$
$$[D,K_{\pm \alpha}]=[K_\alpha,K_\beta]=0$$
$$(\gamma^{\pm 1/2})^2 = K_{\pm \delta} $$
$$ E_iF_j-F_jE_i = \delta_{ij}\frac{K_i - K_i^{-1}}{q_i-q_i^{-1}} $$
$$ K_{\alpha}E_iK_{-\alpha}= q^{(\alpha|\alpha_i)}E_i, \quad K_{\alpha}F_iK_{-\alpha}= q^{-(\alpha|\alpha_i)}F_i$$
$$ DE_iD^{-1} = q^{\delta_{i,0}}E_i, \quad DF_iD^{-1} = q^{-\delta_{i,0}}F_i $$
$$ \sum_{s=0}^{1-a_{ij}} (-1)^s E_i^{(1-a_{ij}-s)}E_jE_i^{(s)} =  \sum_{s=0}^{1-a_{ij}} (-1)^s F_i^{(1-a_{ij}-s)}F_jF_i^{(s)} = 0, \quad i\neq j $$

where $K_i=K_{\alpha_i}$, $K_{\alpha} = \prod_i K_i^{m_i}$ for $\alpha = \sum_i m_i\alpha_i$, $q_i = q^{d_i}$, $[n]_i=\frac{q_i^n - q_i^{-n}}{q_i-q_i^{-1}}$, the quantum factorial is defined by $[n]_i!=[n]_i[n-1]_i\cdots [2]_i[1]_i$, and divided powers given by $E_i^{(s)} = E_i^{s}/[s]_i!$, and $F_i^{(s)} = F_i^{s}/[s]_i!.$\\ 

This algebra possesses a triangular decomposition $U_q(\hat{\g}^{(r)}) = U_q^{-}\otimes U_q^0 \otimes U_q^{+}$, where $U_q^{-}$, $U_q^{+}$, and $U_q^0$ are the subalgebras of $U_q(\hat{\g}^{(r)})$ generated by $\{F_i \ |\ i=0, \ldots, n\}$, $\{E_i \ |\ i=0, \ldots, n\}$, and $\{K_i^{\pm 1}, D^{\pm 1}, \gamma^{\pm 1/2} \ |\ i=0, \ldots, n\}$, respectively. \\

Set $\phi$ and $\Omega$ as the $\C$-linear automorphism and $\C$-linear anti-automorphism, respectively, of $U_q(\hat{\g}^{(r)})$ defined by: 

$$ \phi(E_i) = F_i, \ \ \phi(F_i) = E_i, \ \ \phi(K_i) = K_i, \ \ \phi(D) = D,\ \, \phi(\gamma^{1/2}) = \gamma^{1/2}, \ \ \phi(q) = q^{-1}$$

$$ \Omega(E_i) = F_i, \ \ \Omega(F_i) = E_i, \ \ \Omega(K_i) = K_i^{-1}, \ \ \Omega(D) = D^{-1}, \ \ \Omega(\gamma^{1/2}) = \gamma^{-1/2}, \ \  \Omega(q) = q^{-1}$$

\vspace{3mm}

The braid group $\mc{B}$ associated with the extended affine Weyl group $\tilde{W}$ of $\hat{\g}^{(r)}$ is the group generated by $T_w$, for $w \in \tilde{W}$ with the relations $T_{w}T_{w'} = T_{ww'}$ if $\ell(ww') = \ell(w) + \ell(w')$, where $\ell : \tilde{W} \to \Z_{\geq 0}$ is the length function. Recall that by the results of Iwahori, Matsumoto, and Tits, the generators $T_w$ are independent of the reduced expression of $w$, where such an expression is one of the forms $w=\tau s_{i_1}\ldots s_{i_r}$ for $\ell(w)=r$ and $\tau$ a diagram automorphism. In this case, $T_w = \tau T_{i_1}\ldots T_{i_r}$ (here $T_{i}$ is a shorthand for $T_{s_i}$).\\

The braid group associated with the affine Weyl group $W$ is generated by $T_i$, $i\in I$, and acts as automorphisms on $U_q(\hat{\g}^{(r)})$ according to the following formulas: 

\begin{align*}
    T_i(E_i) = -F_iK_i; & \quad T_i(E_j) = \sum_{r=0}^{-a_{ij}} (-1)^{r-a_{ij}}q_i^{-r} E_i^{(-a_{ij}-r)}E_jE_i^{(r)} \ (i\neq j) \\
    T_i(F_i) = -K_i^{-1}E_i; & \quad T_i(F_j) = \sum_{r=0}^{-a_{ij}} (-1)^{r-a_{ij}}q_i^{r} F_i^{(r)}F_jF_i^{(-a_{ij}-r)} \ (i \neq j)\\
    T_i(K_{\beta}) = K_{s_i(\beta)}; & \quad \quad T_i(D) = DK_i^{-\delta_{i,0}}; \quad \quad T_i(D^{-1}) = D^{-1}K_i^{\delta_{i,0}}
\end{align*}

\vspace{3mm}

In particular, $T_i(K_j) = K_jK_i^{-a_{ij}}$ and $T_i(K_j^{-1}) = K_j^{-1}K_i^{a_{ij}}$. It is also known that $T_i$ satisfies $\Omega T_i = T_i\Omega$ and $\phi T_i = T_i^{-1}\phi$, for all $i\in I$.

\subsection{Root vectors and their commutation relations}

We are going to recall the definition of root vectors and some of their commutation relations. All the presented formulas can be found in \cite{D}; in particular, section 3 of the cited paper gives an exhaustive exposition of the case $A_2^{(2)}$. \\

For all $i \in I_0$, define the element $\varpi_i \in P^{\vee}$ as $\varpi_i := \tilde{d}_i\omega_i$

$$ \varpi_i =  \begin{cases} \omega_i^{\vee} & \mbox{ in the untwisted case or in type} A_{2n}^{(2)} \\
\omega_i & \mbox{otherwise}   
\end{cases}$$

Let $\pi : \Z \to I$ be a map chosen such that

\begin{itemize}
    \item there exists $0 = M_0 < M_1 < \cdots < M_n = M = \ell(\varpi_1\cdots\varpi_n) = \sum_{i=1}^n \ell(\varpi_i) $ (here we use the notation $M_j = \sum_{i=1}^j \ell(\varpi_i)$).
    \item there exists $\tau_i, \ldots, \tau_n \in \mc{T}$ such that $\varpi_i = s_{\pi_{M_{i-1} + 1}} \cdot \ldots \cdot s_{\pi_{M_i}}\tau_i$, for all $i\in I$.
    \item $\pi_{r + M_n} = \tau_1\cdot\cdots\cdot\tau_n(\pi_r)$
\end{itemize}

and let

$$ \beta_r = \begin{cases}
    s_{\pi(0)}s_{\pi(-1)} \cdots s_{\pi(r+1)}(\alpha_{\pi(r)}) & r\leq 0 \\
    s_{\pi(1)}s_{\pi(2)} \cdots s_{\pi(r-1)}(\alpha_{\pi(r)}) & r \geq 1 \\
\end{cases}  $$ 

\vspace{0.3cm}

then, the map $\pi$ induces a bijection $\pi' : \Z \to \Delta_{+}^{re}$ given by $\pi'(r) = \beta_r$ such that $\{\beta_r | r\leq 0\} = \{ \alpha + k\delta | \alpha \in \Delta_{0,+}, k\geq 0 \}$ and $\{\beta_r | r\geq 1\} = \{ -\alpha + k\delta | \alpha \in \Delta_{0,+}, k > 0 \}$, see \cite{B02} and \cite{P}. As a matter of notation, we will consider 

$$ w_r = \begin{cases}
    s_{\pi(0)}s_{\pi(-1)} \cdots s_{\pi(r+1)} & r\leq 0 \\
    s_{\pi(1)}s_{\pi(2)} \cdots s_{\pi(r-1)} & r \geq 1 \\
\end{cases}  $$ 

so that $\beta_r = w_r(\alpha_{\pi(r)})$.\\

The set $\tilde{\Delta}_+$ possesses a (convex\footnote{Here, convex means the following: If $\alpha \prec \beta$ and $\alpha \in \Delta_{+}^{re}$ and $\beta\in \Delta_{+}$ are such that $\alpha + \beta \in \Delta_{+}$, then $\alpha \prec \alpha + \beta \prec \beta$.}) total ordering $\preceq$ defined as follows\footnote{As is usual, we define $x\prec y$ by $x\preceq y$ and $x \neq y$.}: 

\begin{itemize}

\item $ \text{For all } r,s\in \mb{Z}, \quad  \beta_r   \succeq \beta_s  \Leftrightarrow s\leq r\leq 0 \text{ or } r\leq 0 < s \text{ or } 1\leq s \leq r$.

\item $ \text{For all } r\in \mb{Z}, \text{ and for all } \alpha \in \tilde{\Delta}_+^{im}, \quad \beta_r \succeq \alpha \Leftrightarrow r\leq 0 $.

\item $ \text{For all } (r\delta,i), (s\delta,j) \in \tilde{\Delta}_+^{im}, \quad (r\delta,i) \succeq (s\delta,j) \Leftrightarrow r < s \text{ or } r=s, i\geq j $.

\end{itemize}

So, we have

$$ \beta_{0} \succ \beta_{-1} \succ \beta_{-2} \succ \cdots \succ (\delta,N) \succ \cdots (\delta, 1) \succ (2\delta,N) \succ \cdots \succ (r\delta,2) \succ (r\delta, 1) \succ \cdots \succ \beta_2 \succ \beta_1  $$

\vspace{0.3cm} 

and in the root system without multiplicities

$$ \beta_{0} \succ \beta_{-1} \succ \beta_{-2} \succ \cdots \succ \delta \succ 2\delta \succ \cdots \succ \beta_2 \succ \beta_1  $$

\vspace{0.3cm} 

Note that we inverted the original ordering defined in \cite{D}. We also say that two sets $A$ and $B$ satisfy the relation $A \prec B$ if $x\prec y$ for any $x\in A$ and $y\in B$. With this notation, we have that the set of positive roots is divided into three sets: 

$$ \{ \alpha + k\delta | \alpha \in \Delta_{0,+}, k\geq 0 \} \succ \{ k\delta | k > 0 \} \succ \{ -\alpha + k\delta | \alpha \in \Delta_{0,+}, k > 0 \} $$

\vspace{0.3cm} 

and, similarly, the set of negative roots is divided as:

$$ \{ -\alpha - k\delta | \alpha \in \Delta_{0,+}, k\geq 0 \} \prec \{ -k\delta | k > 0 \} \prec \{ \alpha - k\delta | \alpha \in \Delta_{0,+}, k > 0 \} $$

\vspace{4mm}

The above definitions, allow us to define root vector as follows. For each $\beta_r = \omega_r(\alpha_{\pi(r)}) \in \Delta_{+}^{re}$, we define the (quantum) positive real root vectors by 

$$ E_{\beta_r} = \begin{cases}
    T_{w_r^{-1}}^{-1}(E_{\alpha(\pi(r))}) = T_{s_{\pi(0)}}^{-1}T_{s_{\pi(-1)}}^{-1}\cdots T_{s_{\pi(r+1)}}^{-1}(E_{\alpha(\pi(r))}) & r\leq 0 \\
    T_{w_r}(E_{\alpha(\pi(r))}) = T_{s_{\pi(1)}}T_{s_{\pi(2)}} \cdots T_{s_{\pi(r-1)}}(E_{\alpha(\pi(r))}) & r \geq 1 \\
\end{cases}  $$ 

\vspace{5mm}

For $i \in I_0$ and $k>0$ let $k\delta$ be a positive imaginary root. Set $\tilde{E}_{k\tilde{d}_i\delta}^{(i)} = K_i^{-1}[E_i, 
K_iE_{k\tilde{d}_i\delta - \alpha_i}]$, and define the $N$ imaginary positive root vectors, $E_{k\delta}^{(i)}$, by the following functional equation: 

$$ \exp\Big( (q_i - q_i^{-1}) \sum_{k = 1}^{\infty} E_{k\tilde{d}_i\delta}^{(i)}\ z^k \Big) = 1 + (q_i - q_i^{-1})\sum_{k=1}^{\infty} \tilde{E}_{k\tilde{d}_i\delta}^{(i)}z^k $$

We also define $F_{k\tilde{d}_i\delta}^{(i)}:=\Omega(E_{k\tilde{d}_i\delta}^{(i)})$.

\begin{rem}
    These imaginary root vectors are the same as the ones defined in \cite{D}, because $\tilde{E}_{k\tilde{d}_i\delta}^{(i)} = K_i^{-1}[E_i, K_iE_{k\tilde{d}_i\delta - \alpha_i}] = K_i^{-1}E_iK_iE_{k\tilde{d}_i\delta - \alpha_i} - E_{k\tilde{d}_i\delta - \alpha_i}E_i = q_i^{-2}E_iE_{k\tilde{d}_i\delta - \alpha_i} - E_{k\tilde{d}_i\delta - \alpha_i}E_i$
\end{rem}

\begin{rem}
    The root vectors $E_{k\tilde{d}_i\delta \pm \alpha_i}$ can be defined in terms of the braid operators as follows: 
    $$ E_{k\tilde{d}_i\delta + \alpha_i} = T^{-k}_{\varpi_i}(E_i), \quad \mbox{ for all } \quad k \in \Z_{\geq 0} $$
    $$ E_{k\tilde{d}_i\delta - \alpha_i} = T^{k}_{\varpi_i}T^{-1}_{i}(E_i) = -T^{k}_{\varpi_i}(K_i^{-1}F_i), \quad \mbox{ for all } \quad k\in \Z_{> 0} $$
\end{rem}

The main commutation relations between root vectors are given in Theorem 5.3.2 of \cite{D}. We recollect them here: Let $i,j \in I_0$ and $r,s > 0$ be such that $\tilde{d}_i|r$ and $\tilde{d}_j|s$. Then we have: 

\begin{equation}\label{for2.1}
    [E_{r\delta}^{(i)}, E_{(s-\tilde{d}_j)\delta + \alpha_j}] = \begin{cases}
        x_{ijr}E_{(s-\tilde{d}_j + r)\delta+\alpha_j} & \mbox{if } \tilde{d}_j|r \\
        0 & \mbox{otherwise}
    \end{cases}
\end{equation}

\begin{equation}\label{for2.2}
    [E_{r\delta}^{(i)}, E_{s\delta + \alpha_j}] = \begin{cases}
        -x_{ijr}E_{(s + r)\delta-\alpha_j} & \mbox{if } \tilde{d}_j|r \\
        0 & \mbox{otherwise}
    \end{cases}
\end{equation}

\begin{equation}\label{commfor}
    [E_{r\delta}^{(i)}, F_{(s-\tilde{d}_j)\delta + \alpha_j}] = \begin{cases}
        x_{ijr}K_{(s-\tilde{d}_j)\delta + \alpha_j}E_{(r - s + \tilde{d}_j)\delta-\alpha_j} & \mbox{if } \tilde{d}_j|r \mbox{ and } r\geq s \\
        -x_{ijr}K_{r\delta}F_{(s-r-\tilde{d}_j)\delta + \alpha_j} & \mbox{if } \tilde{d}_j|r \mbox{ and } r < s\\
        0 & \mbox{otherwise}
    \end{cases}
\end{equation}

\begin{equation}
    [E_{r\delta}^{(i)}, F_{s\delta - \alpha_j}] = \begin{cases}
        -x_{ijr}E_{(r - s)\delta+\alpha_j}K_{\alpha_j-s\delta} & \mbox{if } \tilde{d}_j|r \mbox{ and } r\geq s \\
        x_{ijr}F_{(s-r)\delta - \alpha_j}K^{-1}_{r\delta} & \mbox{if } \tilde{d}_j|r \mbox{ and } r < s\\
        0 & \mbox{otherwise}
    \end{cases}
\end{equation}

\begin{equation}\label{for2.5}
    [E_{r\delta}^{(i)}, E_{s\delta}^{(j)}] = 0
\end{equation}

\begin{equation}
    [E_{r\delta}^{(i)}, F_{s\delta}^{(j)}] = \delta_{rs}x_{ijr}\frac{K_{r\delta} - K^{-1}_{r\delta}}{q_j - q_j^{-1}}
\end{equation}

where

\begin{equation*}
    x_{ijr} = \begin{cases}
        (o(i)o(j))^r{\displaystyle\frac{[ra_{ij}]_{q_i}}{r}} & \mbox{in the non-twisted case or case } A_{2n}^{(2)} \mbox{ for } (i,j) \neq (1,1)\\
        {\displaystyle\frac{[2r]_q}{r}}(q^{2r}+(-1)^{r-1}+q^{-2r}) & \mbox{in case } A_{2n}^{(2)}, i=j=1\\
        (o(i)o(j))^{\frac{r}{\tilde{d}_ib_{ij}}}{\displaystyle\frac{\tilde{d}_ib_{ij}[ra_{ij}]_{q}}{r[d_i]_q}} & \mbox{otherwise}
    \end{cases}
\end{equation*}

where $o : I_0 \to \{\pm 1\}$ is a function on the set of vertices of the Dynkin diagram such that if $a_{ij} < 0$ (i.e. the vertices are adjacent), then $o(i)o(j)=-1$, and $b_{ij}$ is 1 if $a_{ij}\geq -1$, or it is $d_j$ in any other case.  \\

Let $i,j\in I_0$ and $r>0$; then $i \neq 1$ and

\begin{equation}
    T_{\varpi_j}(\tilde{E}_{r\tilde{d_i}\delta}^{(i)}) = \tilde{E}_{r\tilde{d_i}\delta}^{(i)}, \quad \quad T_{\varpi_j}(E_{r\tilde{d_i}\delta}^{(i)}) = E_{r\tilde{d_i}\delta}^{(i)}
\end{equation}

These two relations follow from \cite{D} Proposition 2.2.4 item b, for the case $(X_N^{(r)}, i) \neq (A_{2n}^{(2)},1)$, and Proposition 4.4.1 for the case $A_{2n}^{(2)}$ with $i=1$. We conclude this subsection with the following commutation relation:

\begin{align}\label{for2.8}
    [E_{k\delta}^{(i)}, [E_j, F_j]] &=  x_{ijk}\begin{cases}
        [E_{k\delta + \alpha_i}, F_j] + [E_j, K_{\alpha_j}E_{k\delta-\alpha_j}] & \mbox{if } \tilde{d}_j|r\\
        0 & \mbox{ otherwise }
    \end{cases}
\end{align}

which follows from the fact that $[E_{k\delta}^{(i)}, [E_j, F_j]] = [[E_{k\delta}^{(i)}, E_j], F_j] + [E_j, [E_{k\delta}^{(i)}, F_j]]$ and formulas (\ref{for2.1}) and (\ref{commfor}).

\subsection{A PBW basis for twisted quantum affine algebras} 

We recall the construction of the PBW basis for twisted quantum affine algebras presented in \cite{D}. We do not provide any proofs and refer the reader to the cited article and the references therein.\\

For any $\eta \in Q$ we define the set 
$$\mc{P}(\eta) = \{\underline{\gamma} = (\gamma_1, \ldots, \gamma_r) \in \bigcup_{r\in \mb{N}} \tilde{\Delta}^r_{+} | \gamma_1 \preceq \cdots \preceq \gamma_r, \sum_{u=1}^{r} p(\gamma_u) = \eta \}$$ 

and consider the set $\mc{P} := \bigcup_{\eta \in Q} \mc{P}(\eta)$ with the induced lexicographical order, denoted also by $\preceq$ and defined as follows: let $\underline{\gamma} = (\gamma_1, \ldots, \gamma_r), \underline{\gamma}' = (\gamma_1', \ldots, \gamma_s') \in \mc{P}$, we say that $\underline{\gamma} \preceq \underline{\gamma}'$ if and only if $l = \min\{u= 1, \ldots, r | \gamma_u \neq \gamma_u'\}$ then $\gamma_l \prec \gamma_l'$.\\

Given an arbitrary function $x: \tilde{\Delta}_{+} \to U_q(\hat{\g}^{(r)})$, $\alpha\mapsto x_\alpha$, we extend it to $\mc{P}$ by the rule $x(\underline{\gamma}) := x_{\gamma_1}\cdots x_{\gamma_r}$, where $\underline{\gamma} = (\gamma_1, \ldots, \gamma_r) \in \mc{P}$. With this notation, we can establish the following.

\begin{thm}\label{TwistPBW}
    The set $\{ E(\underline{\gamma}) \ | \ \underline{\gamma} \in \mc{P} \}$ is a $\C(q)$-basis of $U_q^{+}(\hat{\g}^{(r)})$. Moreover, the set $\{ F(\underline{\gamma}) K_\lambda E(\underline{\gamma'}) \ |\ \lambda\in Q, \underline{\gamma}, \underline{\gamma}'\in \mc{P} \}$ is a $\C(q)$-basis of $U_q(\hat{\g}^{(r)})$. Furthermore, for $\alpha, \beta \in \tilde{\Delta}_+$ such that $\alpha \succ \beta$ then
    $$ E_{\alpha}E_{\beta} - q^{(p(\alpha) | p(\beta))}E_{\beta}E_{\alpha} = \sum_{\beta \prec \gamma_1 \prec \cdots \prec \gamma_r \prec \alpha} c_{\underline{\gamma}}E(\underline{\gamma}) $$
    for $\underline{\gamma} = (\gamma_1, \ldots, \gamma_r)$ and $c_{\underline{\gamma}} \in \C[q,q^{-1}]$.
    
\end{thm}

\dem See Theorem 6.2.2 and Theorem 6.2.3 of \cite{D}. \findem

To be more explicit, let us consider the following sets of positive and negative roots:

\begin{align*}
    A_1 &= \{ \alpha + k \delta | \alpha \in \Delta_{0,+}, k\geq 0 \} \\
    A_2 &= \{ k\delta | k\geq 0 \} \\
    A_3 &= \{ -\alpha + k \delta | \alpha \in \Delta_{0,+}, k > 0 \} \\
    B_1 &= \{ -\alpha - k \delta | \alpha \in \Delta_{0,+}, k\geq 0 \} \\
    B_2 &= \{ -k\delta | k > 0 \} \\
    B_3 &= \{ \alpha - k \delta | \alpha \in \Delta_{0,+}, k > 0 \} \\
\end{align*}

They satisfy 

$$ A_1 \succ A_2 \succ A_3 \succ B_3 \succ B_2 \succ B_1 $$

Let $X_i$ ($i = 1,2,3$) be a monomial of the form $E_{\gamma_i^1}\cdots E_{\gamma_i^r}$ for $\gamma_i^j \in A_i$, $j=1,\ldots, r$, and $\gamma_i^1 \preceq \cdots \preceq \gamma_i^r$. Let $Y_i$ ($i = 1,2,3$) be a monomial of the form $E_{\eta_i^1}\cdots E_{\eta_i^s}$, for $\eta_i^l \in B_i$, $l=1,\ldots, s$, and $\eta_i^1 \preceq \cdots \preceq \eta_i^s$. And let $Z$ be a product of elements in $U_q^0(\hat{\g}^{(r)})$. Then, we can rephrase Theorem \ref{TwistPBW} as follows.

\begin{thm}\label{PBWBasisbasic}
    The set $\{ X_3X_2X_1 \  | \ X_i \in A_i, \ i=1,2,3\}$ is a $\C(q)$-basis of $U_q^{+}(\hat{\g}^{(r)})$. Moreover, the set $\{ Y_1Y_2Y_3ZX_3X_2X_1 \  | \ X_i \in A_i, Y_i\in B_i, Z \in  U_q^0(\hat{\g}^{(r)}), \ i= 1,2,3 \}$ is a $\C(q)$-basis of $U_q(\hat{\g}^{(r)})$.
\end{thm}

\dem Immediate from Theorem \ref{TwistPBW} together with the total order $\preceq$ described above.

\findem

We can construct also a grading by degree on $U_q(\hat{\g}^{(r)})$, similarly as was done by Beck and Kac in \cite{BK}. Any element of the form $X_3X_2X_1$ can be written as $\prod_{\beta\in \tilde{\Delta}_{+}} E_{\beta}^{a'_\beta}$ for $a'_{\beta} \in \Z_{+}$, if this is the case we will denote the element by $E_{(a'_\beta)}$, where $(a'_\beta) \in \Z^{\tilde{\Delta}_{+}}$. The same is done for elements of the form $Y_1Y_2Y_3$. Set $K$ to be a product of elements in $U_q^0(\hat{\g}^{(r)})$. The total height of a monomial $F_{(a_\beta)}K E_{(a'_\beta)}$ is given by 

$$ d_0(F_{(a_\beta)}K E_{(a'_\beta)}) = \sum_{\beta \in \tilde{\Delta}_{+}} (a_\beta + a'_\beta) \mathrm{ht}(\beta)$$

and its total height degree is 

$$ d(F_{(a_\beta)}K E_{(a'_\beta)}) = (d_0(F_{(a_\beta)}K E_{(a'_\beta)}), (a_\beta), (a'_\beta)) \in \Z^{2\tilde{\Delta}_{+} + 1} $$

Next, we consider $\Z^{2\tilde{\Delta}_{+} + 1}$ as a totally ordered semigroup with the usual lexicographical order. We define $U_s$, for any $s\in \Z^{2\tilde{\Delta}_{+} + 1}$, as the span of the basis monomials $F_{(a_\beta)}KE_{(a'_{\beta})}$ with $d(F_{(a_\beta)}KE_{(a'_{\beta})}) \leq s$. Let $Gr U_q(\hat{\g}^{(r)})$ be the associated graded algebra with respect to the above filtration.

\begin{prop}\label{KacBeckCommRel} In the algebra $Gr U_q(\hat{\g}^{(r)})$, we have the following commutation relations:
    \begin{align*}
    K_{\alpha}K_\beta = K_{\alpha + \beta} & \quad; \quad K_0 = 1\\
    K_\alpha E_\beta  = q^{(\alpha | \beta)} E_\beta K_\alpha & \quad; \quad E_\alpha E_{-\beta} = E_{-\beta} E_\alpha, \quad \mbox{ for } \alpha,\beta \in \tilde{\Delta}_{+}\\
    E_{\alpha}E_{\beta} = q^{(\alpha | \beta)}E_{\beta}E_{\alpha} & \quad ; \quad E_{-\alpha}E_{-\beta} = q^{(\alpha | \beta)}E_{-\beta}E_{-\alpha} \mbox{ for } \alpha,\beta \in \tilde{\Delta}_{+}, \beta \prec \alpha
\end{align*}
\end{prop}

\dem It follows from Theorem \ref{TwistPBW} and Proposition 1.8 in \cite{BK}. \findem

\section{Imaginary Verma modules}

Consider the natural partition $\Delta = S \cup (-S)$ where  $S= \{ \alpha + k\delta | \alpha\in \Delta_{0,+}, k\in \Z   \} \cup \{ k\delta | k>0   \}$. We are going to define Verma-like modules for this partition, which, in this case, are called twisted imaginary Verma modules. \\

Let $U_q(\pm S)$ be the subalgebra of $U_q(\hat{\g}^{(r)})$ generated by $\{E_{\beta} | \beta \in \pm S\}$, i.e., $\{ E_{\alpha + k\delta} | \alpha\in \Delta_{0,+}, k\in \Z   \} \cup \{ E_{k\delta}^{(i)} | k>0, \ 1 \leq i \leq N   \}$. Let $B_q$ be the subalgebra of $U_q(\hat{\g}^{(r)})$ generated by $\{E_{\beta} | \beta \in S\} \cup U_q^0(\hat{\g})$.\\

For $\lambda\in P$, (integral weight lattice) we say that a weight module $V_q$ for $U_q(\hat{\g}^{(r)})$ is an $S$-highest weight module with highest weight $\lambda$ if there is some non-zero vector $v\in V_q$ such that $u\cdot v = 0$ for all $u \in U_q(+S) \setminus (\C(q)^{*}\cdot 1)$ and $V=U_q(\hat{\g}^{(r)})\cdot v$.\\

We make $\C(q)$ a 1-dimensional $B_q$-module of weight $\lambda \in P$ by picking a generator $v$ and setting $E_\beta \cdot v = 0$ for $\beta \in S$, $K_i^{\pm}\cdot v = q_i^{\pm \lambda(h_i)}v$ for all $i\in I$, $\gamma^{\pm 1/2}\cdot v = q^{\pm \lambda(c)/2}v$, and $D^{\pm 1}\cdot v = q^{\lambda(d)}v$.

\begin{defi}
    The induced module $M_q(\lambda) := U_q(\hat{\g}^{(r)}) \otimes_{B_q} \C(q)\cdot v$ is called a twisted quantum imaginary Verma module.
\end{defi}

The module $M_q(\lambda)$ is an $S$-highest weight $U_q(\hat{\g}^{(r)})$-module, but we cannot say that it is isomorphic to $U_q^{-}(S)\cdot v$ since we do not have a PBW theorem for imaginary partitions associated with twisted quantum affine algebras. However, we can construct a particular basis of $M_q(\lambda)$ as a $\C(q)$-vector space.

\begin{thm}\label{basisMq(l)}
    As a $\C(q)$-vector space, the twisted imaginary Verma module $M_q(\lambda)$ is isomorphic to the space spanned by the ordered monomials $\prod_i E_{-\alpha_i-l_i\delta} \prod_j E_{-k\delta}^{n_k} \prod_s E_{-\beta_s + m_s\delta}$ for $\alpha_i, \beta_s \in \Delta_{0,+}$, $l_i\geq 0$, $k_j, m_s>0$.
\end{thm}

\dem Let $v_{\lambda}$ be the canonical generator of $M_q(\lambda)$ as an $S$-highest weight $U_q(\hat{\g}^{(r)})$-module such that $M_q(\lambda) = U_q(\hat{\g}^{(r)})\cdot v_{\lambda}$. Then, for arbitrary $w \in M_q(\lambda)$ we will have $w = u\cdot v_{\lambda}$ for some $u\in U_q(\hat{\g}^{(r)})$, which can be written as $u = \sum Y_1Y_2Y_3ZX_3X_2X_1$ thanks to Theorem \ref{PBWBasisbasic}.\\

By definition of $M_q(\lambda)$, we have that $X_1\cdot v_\lambda = 0 = X_2\cdot v_\lambda$. Moreover, $Z$ commutes with $Y_3$, up to scalars, and acts on $v_{\lambda}$ by a scalar. So, we have that $w$ has the form $w = \sum Y_1Y_2Y_3X_3 \cdot v_\lambda$. We are going to show that the generating monomials of $M_q(\lambda)$ are of the form $Y_1Y_2X_3\cdot v_{\lambda}$ by induction on the total height degree of $X_3$.\\

First of all, note that the elements of the form $Y_3$ act by zero on $v_\lambda$. Let us consider an element $Y_3X_3$, which is of the form $\cdots E_{\alpha-l\delta}E_{-\beta + m\delta}\cdots$ for $l,m >0$ and $\alpha, \beta \in \Delta_{0,+}$. So, thanks to Proposition \ref{KacBeckCommRel}, we have that 

\begin{equation}\label{indheight}
    E_{\alpha-l\delta}E_{-\beta + m\delta} = q^{-(-\beta + m\delta | \alpha-l\delta)} E_{-\beta + m\delta}E_{\alpha-l\delta} + \sum x_{\gamma}E_{\gamma}
\end{equation}

for $x_\gamma\in \C[q,q^{-1}]$ scalars, and the elements $E_{\gamma}$ of lower degrees with $-\beta + m\delta \prec \gamma \prec \alpha -l\delta$.\\

If $d_0(X_3)=1$, then there is just one simple root involved in it, and there are no lower terms in equation \ref{indheight}, so we are done. Hence we will assume the theorem is true for $X_3$ up to total height degree $d$.\\

Let $X_3$ be of the form $E_{-\beta_1+m_1\delta} \cdots E_{-\beta_s + m_s\delta}$ and let $Y_3$ of the form $E_{\alpha_1-n_1\delta} \cdots E_{\alpha_\ell - n_\ell\delta}$ for roots $-\beta_i + m_i\delta \in A_3$ and $\alpha_j - n_i\delta \in B_3$. Then, by Theorem \ref{TwistPBW} we have 

\begin{align*}
    Y_3X_3 &= E_{-\beta_1+m_1\delta} \cdots E_{-\beta_s + m_s\delta}E_{\alpha_1-n_1\delta} \cdots E_{\alpha_\ell - n_\ell\delta} \\
    &= E_{-\beta_1+m_1\delta} \cdots E_{\alpha_1-n_1\delta}E_{-\beta_s + m_s\delta} \cdots E_{\alpha_\ell - n_\ell\delta} \\
    &+ \sum_{-\beta_s + m_s\delta \prec \gamma_1 \prec \cdots \prec \gamma_r \prec \alpha_1-n_1\delta} x_\gamma E_{\gamma_1} \cdots E_{\gamma_r}
\end{align*}

where the terms in the sum are of lower degrees. Iterating this process, we get that 

$$ Y_3X_3 = X_3Y_3 + (\mbox{lower degree terms}) $$

And so, acting on $v_\lambda$ we get an expression of the form 

$$ Y_3X_3\cdot v_\lambda = (\mbox{lower degree terms})\cdot v_\lambda $$

By induction, the lower degree terms can be reordered in the desired form, getting rid of the terms of the form $Y_3$, because they act on $v_\lambda$ by zero, and just surviving the expression consisting of terms of the form $X_3$.\\

The above arguments then prove that if  $w = \sum Y_1Y_2Y_3X_3 \cdot v_\lambda$ we can reduce the monomials in the expression to get $w = \sum Y_1Y_2X_3 \cdot v_\lambda$ as desired.

\findem

\section{$\mathbb{A}$-forms and Classical limits}

In this section, we construct an $\mathbb{A}$-form for the twisted quantum affine algebra and its imaginary Verma modules. We apply this to compute its classical limits, showing that they coincide with imaginary Verma modules associated with the twisted affine Lie algebra.\\ 

Recall that for any $i\in I_0$, $s\in \Z$, and $k\in \Z^{+}$, the Lusztig elements are defined by 

$$ \begin{bmatrix} K_i \ ;\  s \\ k \end{bmatrix}  = \prod_{r=1}^k \frac{K_iq_i^{s-r+1} - K_i^{-1}q_i^{r-s-1}}{q_i^{r} - q_i^{-r}}$$

$$ \begin{bmatrix} D \ ;\  s \\ k \end{bmatrix}  = \prod_{r=1}^k \frac{Dq_0^{s-r+1} - D^{-1}q_0^{r-s-1}}{q_0^{r} - q_0^{-r}}$$

Let $\mathbb{A} = \C[q, q^{-1}, \frac{1}{[k]_{q_i}}, i \in I_0, k>0]$ and define the $\A$-form $U_{\A}$ of $U_q(\hat{\g}^{(r)})$ as the $\A$-subalgebra of $U_q(\hat{\g}^{(r)})$ generated by the elements $E_i$, $F_i$, $K_\alpha^{\pm 1}$, $\begin{bmatrix} K_i \ ;\  0 \\ 1 \end{bmatrix}$ and $\begin{bmatrix} D \ ;\  0 \\ 1 \end{bmatrix}$. It is known that all Lusztig elements belong to $U_{\A}$ and for all $i,j \in I_0$, $s\in \Z$ and $n\in \Z^{+}$ they satisfy the following relations (see \cite{FGM}, Proposition 4.1)

\begin{align*}
    E_i\begin{bmatrix} K_j \ ;\  s \\ k \end{bmatrix} = \begin{bmatrix} K_j \ ;\  s-a_{ij} \\ k \end{bmatrix}E_i \\
    \begin{bmatrix} K_j \ ;\  s \\ k \end{bmatrix}F_i = F_i \begin{bmatrix} K_j \ ;\  s-a_{ij} \\ k \end{bmatrix} \\
    E_i\begin{bmatrix} D \ ;\  s \\ k \end{bmatrix} = \begin{bmatrix} D \ ;\  s-\delta_{i,0} \\ k \end{bmatrix}E_i \\
    \begin{bmatrix} D \ ;\  s \\ k \end{bmatrix}F_i = F_i \begin{bmatrix} D \ ;\  s-\delta_{i,0} \\ k \end{bmatrix} \\
    E_iF_j \neq F_jE_i, \quad \mbox{ for } i\neq j\\
    E_iF_i^n = F_i^nE_i + F_i^{n-1}\sum_{r=0}^{n-1}\begin{bmatrix} K_i \ ;\  -2r \\ 1 \end{bmatrix}
\end{align*}

Set $U_{\A}^{+}$ and $U_{\A}^{-}$ to be the subalgebras of $U_{\A}$ generated by $E_i$ and $F_i$, $i\in I_0$, respectively. Let $U^0_{\A}$ be the subalgebra of $U_{\A}$ generated by the elements $K_i^{\pm}$ ($i\in I_0$), $D^{\pm}$, $\begin{bmatrix} K_i \ ;\  0 \\ 1 \end{bmatrix}$ and $\begin{bmatrix} D \ ;\  0 \\ 1 \end{bmatrix}$.   

\begin{prop}
    The subalgebra $U_{\A}$ has the triangular decomposition $U_{\A} = U_{\A}^{-}\otimes U_{\A}^0 \otimes U_{\A}^{+}$.
\end{prop}

\dem It follows from the above relations on generators. \findem

For $\lambda \in P$ let $M_q(\lambda)$ the twisted imaginary Verma module with generator $v_\lambda$. We define $M_{\A}(\lambda) := U_{\A}\cdot v_{\lambda}$ and we call it the $\A$-form $M_{\A}(\lambda)$. In other words, it is the $U_{\A}$-submodule of $M_q(\lambda)$ generated by $v_\lambda$.

\begin{prop}\label{BasesM_A}
    As a vector space over $\A$, the module $M_{\A}(\lambda)$ is spanned by the ordered monomials $\prod_i E_{-\alpha_i-l_i\delta} \prod_j E_{-k\delta}^{n_k} \prod_s E_{-\beta_s + m_s\delta}$ for $\alpha_i, \beta_s \in \Delta_{0,+}$, $l_i\geq 0$, $k_j, m_s>0$.
\end{prop}

\dem The proof follows the same line of reasoning as the proof of Theorem \ref{basisMq(l)}. We just need to check that the coefficients belong to the ring $\A$. Let $u\in U_{\A}$ be an arbitrary element. Then we can write it as a sum $u = \sum Y_1Y_2Y_3ZX_3X_2X_1$ where $X_i$ and $Y_i$ are as in Theorem \ref{basisMq(l)} but now $Z$ is in $U_{\A}^0$.\\

Consider the element $u\cdot v_\lambda \in M_{\A}(\lambda)$. By definition, we have that $X_1\cdot v = X_2 \cdot v = 0$. Moreover, since $Z\in U_{\A}^0$, by the above commuting relations, it commutes with $X_3$ and we get $u \cdot v_\lambda = \sum Y_1Y_2Y_3X_3Z\cdot v_\lambda$. Because $K_i^{\pm 1}\cdot v_\lambda = q^{\pm \lambda(h_i)}v_\lambda$, $D^{\pm 1}\cdot v_\lambda = q^{\pm \lambda(d)}v_\lambda$,  $\begin{bmatrix} K_i \ ;\  s \\ k \end{bmatrix}   \cdot v_\lambda = \begin{bmatrix} \lambda(h_i) + s\\ k \end{bmatrix}_{q_i} v_\lambda$ and $\begin{bmatrix} D \ ;\  s \\ k \end{bmatrix}   \cdot v_\lambda = \begin{bmatrix} \lambda(d) + s\\ k \end{bmatrix}_{q_0} v_\lambda$, and all the resulting coefficients are in the ring $\A$, we have that $Z\cdot v_\lambda \in \A\cdot v_{\lambda}$.\\

The rest of the proof follows that of Theorem \ref{basisMq(l)}, just noticing that the coefficients that appear in lower degree terms are all in $\A$.

\findem

\begin{prop}\label{CoeffM_AM_q}
    As $\C(q)$-vector spaces, $\C(q)\otimes_{\A}M_{\A}(\lambda) \cong M_q(\lambda)$.
\end{prop}

\dem From left to right, the map is given by $f(q) \otimes m \mapsto f(q)m$. From the other direction, if $B$ is a basis element of $M_q(\lambda)$ of the form $\prod_i E_{-\alpha_i-l_i\delta} \prod_j E_{-k\delta}^{n_k} \prod_s E_{-\beta_s + m_s\delta}$ for $\alpha_i, \beta_s \in \Delta_{0,+}$, $l_i\geq 0$, $k_j, m_s>0$, and is given by Theorem \ref{basisMq(l)}. Define $B\cdot v_\lambda \mapsto 1\otimes B\cdot v_\lambda$. This last map is well defined thanks to Proposition \ref{BasesM_A} and it is easy to see that both maps are mutually inverse. 

\findem

For the module $M_{\A}(\lambda)$ we have a weight structure defined by $M_{\A}(\lambda)_{\mu} = M_{\A}(\lambda)\cap M_q(\lambda)_\mu$. By the above Proposition, we have that, as vector spaces, $M_q(\lambda)_\mu \cong \C(q)\otimes_{\A} M_{\A}(\lambda)_{\mu}$.

\begin{prop}
    $M_{\A}(\lambda) = \displaystyle \bigoplus_{\mu \in P} M_{\A}(\lambda)_\mu$.
\end{prop}

\dem The same proof for the Proposition 3.23 of \cite{BKMe} works. For the sake of completeness, we write the key steps in the proof. Let $v = v_1 + \cdots + v_p \in M_{\A}(\lambda)$ where each $v_j \in M_{q}(\lambda)_{\mu_j}$ for some weights $\mu_j \in P$. Fix an index $i$ and because $\mu_k \neq \mu_i$ there exists an index $i_k$ such that $\mu_k(h_{i_k}) \neq \mu_i(h_{i_k})$. Let $J = \{i_k | k \neq i\}$ and set $s$ such that $s\geq |\mu_k(h_j) - \mu_i(h_j)|$ and $s\geq |\mu_k(d) - \mu_i(d)|$ for all $j\in J$ and $k\neq i$. Set $u\in U_{\A}$ defined as follows

\[ u =  \prod_{j \in J} \begin{bmatrix} K_j \ ;\  -\mu_i(h_j) +s \\ s \end{bmatrix}  \begin{bmatrix} K_j \ ;\  -\mu_i(h_j) - 1 \\ s \end{bmatrix}  \begin{bmatrix} D \ ;\  -\mu_i(d) +s \\ s \end{bmatrix} \begin{bmatrix} D \ ;\  -\mu_i(d) -1 \\ s \end{bmatrix} \]

Then $uv_i = v_i$ and $uv_k = 0$ for $k\neq i$. Since $u \in U_{\A}$, we get that $v_i = uv_i \in M_{\A}(\lambda)$, so $v_i \in M_{\A}(\lambda)_{\mu_i}$.
\findem

Finally, we have the following

\begin{prop}
    For each $\mu \in P$, $M_{\A}(\lambda)_\mu$ is a free $\A$-module and $\rank_{\A} M_{\A}(\lambda)_\mu = \dim_{\C(q)} M_q(\lambda)_\mu$.
\end{prop}

\dem Follows from the above weight decomposition and Proposition \ref{CoeffM_AM_q}.
\findem

We conclude this section with the construction of classical limits of imaginary Verma modules for twisted quantum affine algebras. We are going to show that they are isomorphic to imaginary Verma modules for twisted affine algebras. \\

First of all, let $I=\langle q-1 \rangle{\A}$ be the ideal in $\A$ generated by $q-1$. Recall that there is an isomorphism of fields $\A / I \cong \C$ given by $f + I \mapsto f(1)$. Let $U' = \A / I \otimes_{\A} U_{\A} \cong U_{\A}/IU_{\A}$ and $\bar{U} = U'/K'$ where $K'$ is the ideal of $U'$ generated by $D-1$ and $K_i-1$ for $i\in \{0, 1, \ldots, N\}$. Then, $\overline{U} \cong U(\hat{\g})$ and we have the composition 

$$\xymatrix{ U_{\A} \ar[r] & U_{\A}/ IU_{\A} \cong U' \ar[r] & \overline{U} = U'/K' \cong U(\hat{\g}) }$$

which is called the classical limit. \\

For $u \in U_{\A}$, denote $\overline{u}$ its image in $\overline{U}$. Under this isomorphism, the elements $\overline{E_i}, \overline{F_i}, \overline{D}$ and $\overline{H_i}$, where $H_i = \begin{bmatrix} K_i \ ;\  0 \\ 1 \end{bmatrix}$ may be identified with $e_i, f_i, d$ and $h_i$. Moreover, the classical limit $\overline{E_\beta}$ of a root vector $E_\beta = T_{w}(E_i)$, where $w(\alpha_i) = \beta$, is the root vector $e_\beta = T_{w}(e_i)$ for $U(\hat{\g})$. The same happens for the imaginary root vectors because they are defined using braid operators and root vectors (see \cite{D} sections 2.2 - 2.3).\\

\begin{prop}\label{limitsrootPBW}
    Classical limits of root vectors are root vectors. Moreover, the classical limit of the twisted PBW basis given in Theorem \ref{TwistPBW} is a PBW basis for $U(\hat{\g})$.
\end{prop}

\dem The PBW basis of Theorem \ref{TwistPBW} may be identified with a PBW basis for $U(\hat{\g})$. This follows from Section \ref{prelims} and \cite{B01}, because the classical limit of the monomials in the basis clearly generates the PBW basis of the classical enveloping algebra, and they remain linearly independent because the limits in the classical case do.
\findem

We define the classical limit of $M_{\A}(\lambda)$ to be 

$$ \overline{M_{\A}(\lambda)} := \A / I \otimes_{\A} M_{\A}(\lambda) $$

which by construction is a $\overline{U} \cong U(\hat{\g})$-module. Since $M_{\A}(\lambda)$ has a weight decomposition given by $M_{\A}(\lambda) = \bigoplus_{\mu \in P} M_{\A}(\lambda)_\mu$, then setting $\overline{M_{\A}(\lambda)}_\mu := \A / I \otimes_{\A} M_{\A}(\lambda)_\mu$, we must have $\overline{M_{\A}(\lambda)} = \bigoplus_{\mu \in P} \overline{M_{\A}(\lambda)}_\mu$. Moreover, $\dim_{\C}\overline{M_{\A}(\lambda)}_{\mu} = \rank_{\A} M_{\A}(\lambda)_\mu$ (see Proposition 5.1 in \cite{FGM}).

\begin{prop}
    $\overline{M_{\A}(\lambda)} \cong M(\lambda)$ as $U(\hat{\g})$-modules.
\end{prop}

\dem Thanks to the PBW basis given in Theorem \ref{TwistPBW}, the proof follows word by word the ones in Propositions 5.3 and 5.4 in \cite{FGM}.
\findem

\section{Structure of twisted quantum imaginary Verma modules}

In this section, we are going to describe more closely the structure of quantum imaginary Verma modules. \\

Let $G_q$ be the (Heisenberg) subalgebra of $U_q(\hat{\g}^{(r)})$ generated by elements of the form $X_2 \in A_2$ and $Y_2 \in B_2$. Let $M_q(\lambda)$ be the twisted quantum imaginary Verma module over $U_q(\hat{\g}^{(r)})$ with $S$-highest weight $\lambda$ and generating vector $v_\lambda$. Consider the $G_q$-module $H_q(\lambda) := G_q\cdot v_\lambda$ of $M_q(\lambda)$. It is the module consisting of the finite weight spaces of $M_q(\lambda)$, so $H_q(\lambda) := \sum_{k=0}^{\infty} M_{\lambda - k\delta}$, because elements of $G_q$ which belong to $A_2$ act by zero on $v_\lambda$.

\begin{lem}\label{lemirrHq}
    The $G_q$-module $H_q(\lambda)$ is irreducible if and only if $\lambda(c) \neq 0$.
\end{lem}

\dem The same proof of Proposition 6.2 in \cite{FGM}, taking advantage of Proposition \ref{limitsrootPBW}, implies the Lemma.
\findem

Any $X \in U_q(\hat{\g}^{(r)})$ can be written in the form $X = \sum E_{-\alpha_i-l_i\delta} \cdots E_{-k_i\delta} \cdots E_{-\beta_i + m_i\delta}$ for $\alpha_i, \beta_i \in \Delta_{0,+}$, $l_i\geq 0$, $k_i, m_i>0$ thanks to Theorem \ref{TwistPBW}. Define for $v_\lambda \in M_q(\lambda)$

$$ \tilde{h}(Xv_\lambda) = \sum_{i}( ht(\alpha_i) + ht(\beta_i)) $$

\begin{lem}\label{AuxLem}
    For any $X \in U_q(\hat{\g}^{(r)})$ with $\tilde{h}(Xv_\lambda)>0$, there exists $Y \in U_q(\hat{\g}^{(r)})$ such that $\tilde{h}(YXv_\lambda) < \tilde{h}(Xv_\lambda)$.
\end{lem}

\dem Follows from Theorem \ref{TwistPBW}, Proposition \ref{CoeffM_AM_q} and the proof of Theorem 6.4 in \cite{FGM}.
\findem

\begin{thm}\label{Intersection}
    Let $W$ be a non-trivial submodule of $M_q(\lambda)$. Then $W\cap H_q(\lambda) \neq \{0\}$.
\end{thm}

\dem Let $n \in W$ and let $X \in U_q(\hat{\g}^{(r)})$ such that $Xv_\lambda = n$. If $\tilde{h}(Xv_\lambda) = 0$ then $Xv_\lambda \in H_q(\lambda)$. If $\tilde{h}(Xv_\lambda) >0$, by Lemma \ref{AuxLem}, there exists $Y \in U(\hat{\g})$ such that $\tilde{h}(YXv_\lambda) < \tilde{h}(Xv_\lambda)$, and the statement follows by induction. 
\findem

\begin{cor}
    $M_q(\lambda)$ is irreducible if and only if $\lambda(c)\neq0$. 
\end{cor} 

\dem Follows from Lemma \ref{lemirrHq} and Theorem \ref{Intersection}.
\findem 

Let us suppose that $M_q(\lambda)$ is reducible, that means $\lambda(c)=0$ and so $H_q(\lambda)$ is reducible as a $G_q$-module. Let $M_H$ be the maximal submodule of $H_q(\lambda)$, then $M_H = \sum_{k=1}^{\infty} M_{\lambda - k\delta}$. Denote $M_0 := U_q(\hat{\g}^{(r)})M_H$ and set 

$$ \tilde{M}_q(\lambda) = M_q(\lambda )/M_0$$

It is a quotient of $M_q(\lambda)$ and we call it {\it reduced twisted imaginary Verma module}. We have the following theorem.

\begin{thm} Let $\lambda\in P$ such that $\lambda(c)=0$. Then $\tilde{M}_q(\lambda)$ is irreducible if and only if $\lambda(h_i)\neq 0$ for all $i\in I_0$.  \end{thm}

\dem Suppose that $\lambda(h_i)\neq 0$ for all $i\in I_0$ and let $\tilde{M}_{\A}(\lambda) = \tilde{M}_q(\lambda) \cap M_{\A}(\lambda)$ be the $\A$-form of $\tilde{M}_q(\lambda)$. If $W_q$ is a proper submodule of $\tilde{M}_q(\lambda)$, then its $\A$-form, $W_{\A} = W_q\cap \tilde{M}_{\A}(\lambda)$ is a proper submodule of $\tilde{M}_{\A}(\lambda)$. Taking classical limits, we obtain a proper submodule of $\tilde{M}(\lambda)$, which is not possible according to Theorem 1 \cite{Fut}.\\

On the other hand, let us assume there is an $i \in I_0$ such that $\lambda(h_i)=0$. Let $F_i = E_{-\alpha_i}$ and consider the module $W = U_q(\hat{\g}^{(r)})\cdot F_i\cdot v_\lambda$, where $\tilde{M}_q(\lambda) = U_q(\hat{\g}^{(r)})\cdot v_\lambda$. Clearly $W \neq (0)$ because $F_i\cdot v_\lambda \neq 0$, and if $\tilde{M}_q(\lambda)$ were irreducible, we should have $W = \tilde{M}_q(\lambda)$. If this is the case, there exist elements $c_j \in U_q(\hat{\g}^{(r)})$ such that $v_\lambda = \sum_j c_jF_i\cdot v_\lambda$ and $ht(c_j)=\alpha_i$.\\

By Theorem \ref{PBWBasisbasic}, each $c_j$ is of the form $Y_1Y_2Y_3X_3X_2X_1$. To see the action of this element in $F_i\cdot v_\lambda$ we claim the following: \\

{\it Claim\footnote{This is Lemma 7.2 in \cite{FGM}, there appears a minus sign before $\beta$, but this is a typo.}:} For all $k \in \Z$ and $\beta  \in \Delta_{0,+}$, we have that $E_{\beta+k\delta}F_i\cdot v_\lambda = 0$.\\

By {\it Claim}, $X_1F_iv_\lambda = 0$. Because $X_2\cdot v_\lambda = 0$ and thanks to Formula (\ref{commfor}), we find that $X_2F_i\cdot v_\lambda$ could have three possible forms: it is either zero, or it takes the form $ZX_3$ or $ZY_3$. In any case, since $Y_3\cdot v_{\lambda} = 0$, elements of the form $X_3$ and $Y_3$ commute, and the elements of the form $Z$ act by scalars on $v_\lambda$ and commute with $X_3$, we reduce the $c_j$ to the form $Y_1Y_2X_3$, so we obtain $v_\lambda = \sum_j Y_1Y_2X_3\cdot v_\lambda$. These monomials are non-zero because $F_i$ is of the form $Y_1$, but the height is zero since $v_\lambda = \sum_j c_jF_i\cdot v_\lambda$, which is not possible. Then we arrive at a contradiction; hence, $\tilde{M}_q(\lambda)$ is reducible. \\

{\it Proof of claim:} Follows from Lemma 7.2 in \cite{FGM}, but instead of using the formula $1.6.5d$ of \cite{BK}, use the formula (\ref{for2.8}).
\findem 

\section{Omega operators and their commutation relations}

In this section, we are going to construct the so called Omega operators and prove their main commutation relations. These operators are crucial for the construction of crystal bases for imaginary Verma modules. We start by recalling the Drinfeld realization of the twisted quantum affine algebra following \cite{D01, JM, J}.\\

Let $\omega$ be an $r$-th root of unity; if $j$ belongs to the $\sigma$-orbit of $i$, let $[k]_j = \frac{q_i^k - q_i^{-k}}{q_i - q_i}$. The algebra $U_q(\hat{\g}^{(r)})$, for $r=2,3$, is isomorphic to the associative unital $\C(q^{1/2})$-algebra generated by $x_{ik}^{\pm }$, $h_{il}$, $K_i^{\pm 1}$, $D^{\pm 1}$, and $\gamma^{\pm 1/2}$ for $i = 1, \ldots, N$, $k,l\in \Z$, $l\neq 0$ and subject  to the following relations: 

$$ x_{\sigma(i)r}^{\pm } = \omega^rx_{ir}^{\pm }, \quad \quad h_{\sigma(i)s} = \omega^sh_{is}, \quad K_{\sigma(i)}^{\pm 1} =  K_i^{\pm 1} $$
$$ D^{\pm 1}D^{\mp 1}  = K_i^{\pm 1}K_i^{\mp 1} = \gamma^{\pm 1/2}\gamma^{\mp 1/2} = 1  $$  
$$ [\gamma^{\pm 1/2}, U_q(\hat{\g}^{(r)})] = [D,K_i^{\pm 1}] = [K_i, K_j] = [K_i, h_{js}] = 0    $$
$$ [h_{ik}, h_{jl}] = \delta_{k,-l}\sum_{s=0}^{r-1} \frac{[k \langle \alpha_i' | \sigma^s(\alpha_j') \rangle/d_i]_i}{k}\omega^{ks}\frac{\gamma^k - \gamma^{-k}}{q_j - q_j^{-1}} $$
$$ Dh_{ir}D^{-1} = q^rh_{ir}, \quad \quad  Dx_{ir}^{\pm}D^{-1} = q^rx_{ir}^{\pm}$$
$$ K_ix_{jr}^{\pm}K_i^{-1} = q^{\pm (\alpha_i | \alpha_j )}x_{jr}^{\pm} $$
\begin{equation}\label{comm_h_x_currentform} [h_{ik}, x_{jl}^{\pm}] = \pm \sum_{s=0}^{r-1} \frac{[k \langle \alpha_i' | \sigma^s(\alpha_j') \rangle/d_i]_i}{k}\omega^{ks}\gamma^{\mp |k|/2}x_{j,k+l}^{\pm}  
\end{equation}
\begin{equation}\label{comx_ix_j_currentform} 
(\prod_{s\in \Z/r\Z} (z-\omega^s q^{\pm \langle \alpha_i' | \sigma^s(\alpha_j') \rangle}w))x_{i}^{\pm }(z)x_{j}^{\pm }(w) =  (\prod_{s\in \Z/r\Z} (zq^{\pm \langle \alpha_i' | \sigma^s(\alpha_j') \rangle}- \omega^sw))x_{j}^{\pm }(w)x_{i}^{\pm }(z)
\end{equation}

where 

$$ x_{i}^{\pm }(u) = \sum_{p \in \Z} x_{ip}^{\pm}u^{-p} $$

\begin{equation}\label{Comnx^+x^-} [x_{ik}^{+}, x_{jl}^{-}] = \sum_{s=0}^{r-1}\frac{\delta_{\sigma^s(i),j}\omega^{sl}}{q_i - q_i^{-1}}(\gamma^{(k-l)/2}\psi_{i,k+l} - \gamma^{(l-k)/2}\phi_{i,k+l})  \end{equation}

where

$$ \sum_{k=0}^{\infty} \psi_{ik}z^ k = K_i\exp\Big((q_i-q_i^{-1})\sum_{l>0}h_{il}z^l\Big) $$

$$ \sum_{k=0}^{\infty} \phi_{i,-k}z^{-k} = K_i^{-1}\exp\Big(-(q_i-q_i^{-1})\sum_{l>0}h_{i,-l}z^{-l}\Big) $$

$$ \Sym_{z_1, z_2} P_{ij}^{\pm}(z_1, z_2) \sum_{s=0}^{2}(-1)^s{2 \brack s}_{q^{d_{ij}}} x_{i}(z_1)^{\pm} \cdots x_{i}(z_s)^{\pm}x_{i}(w)^{\pm}x_{i}(z_{s+1})^{\pm} \cdots x_{i}(z_2)^{\pm} = 0$$ 

for $a_{ij} = -1, \sigma(i) \neq j$,

$$ \Sym_{z_1, z_2, z_3} \big[ ((q^{\mp 3r/4}z_1 - q^{r/4} + q^{-r/4})z_2 + q^{\pm 3r/4}z_3 )x_{i}(z_1)^{\pm}x_{i}(z_2)^{\pm}x_{i}(z_3)^{\pm} \big] =0$$

\vspace{3mm}

for $a_{i\sigma(i)} = -1$. Here $\Sym$ means the symmetrization over $z_i$, and $P_{ij}^{\pm}(z,w)$ and $d_{ij}$ are defined as follows:\\

If $\sigma(i)=i$, then $P_{ij}^{\pm}(z,w) = 1$ and $d_{ij} = r$. \\

If $a_{i,\sigma(i)}=0$ and $\sigma(j)=j$, then $P_{ij}^{\pm}(z,w) = \frac{z^rq^{\pm 2r} - w^r}{zq^{\pm 2} - w}$ and $d_{ij} = r$.\\

If $a_{i,\sigma(i)}=0$ and $\sigma(j)\neq j$, then $P_{ij}^{\pm}(z,w) = 1$ and $d_{ij} = 1/2$.\\

If $a_{i,\sigma(i)}=-1$, then $P_{ij}^{\pm}(z,w) = zq^{\pm r/2} + w$ and $d_{ij} = r/4$.\\

\begin{rem}
   Theorem 4 of \cite{D01} claims that $U_q(\hat{\g}^{(r)})$ is isomorphic to the above realization. A complete proof of this fact is found in \cite{J, JM}, Theorems 3.2 and 2.1, respectively. 
\end{rem}

Note that equation (\ref{Comnx^+x^-}) can be written in formal power series form as

\begin{equation}\label{gencommx^+x^-}
    [x_i^+(u), x_j^{-}(v)] = \sum_{s=0}^{r-1} \frac{\delta_{\sigma^s(i)j}}{q_i-q_i^{-1}}\Big(\delta\Big(\frac{u\omega^s}{v\gamma }\Big)\psi_i(v\gamma^{1/2}\omega^{-s}) - \delta\Big(\frac{u\gamma \omega^s}{v }\Big)\phi_i(u\gamma^{1/2})\Big)
\end{equation}

and equation (\ref{comm_h_x_currentform}) is written using the following two identities 

\begin{equation}\label{conm-PhiX_gen}
    \phi_i(u)x_j^{\pm}(v)\phi_i(u)^{-1} = g_{ij,q}(u\gamma^{\mp 1/2}v^{-1})^{\pm}x_j^{\pm}(v)
\end{equation}

\begin{equation}\label{conm-PsiX_gen}
    \psi_i(u)x_j^{\pm}(v)\psi_i(u)^{-1} = g_{ij,q}(v\gamma^{\mp 1/2}u^{-1})^{\mp}x_j^{\pm}(v)
\end{equation}

where $\delta(z) = \sum_{k\in \Z}z^k$ is the formal Dirac delta function, 

$$ \phi_i(u) = \sum_{p\in \Z} \phi_{i,p} u^{-p}, \quad \psi_i(u) = \sum_{p\in \Z} \psi_{i,p} u^{-p}, \quad x_i^{\pm }(v) = \sum_{p\in \Z} x^{\pm}_{i,p} v^{-p} $$

and the function $g_{ij,q}(t)$ is the Taylor expansion at $t=0$ of the function 

$$ \prod_{s\in \Z/r\Z} \frac{ q^{\langle \alpha_i' | \sigma^s(\alpha_j') \rangle}t-\omega^s}{ t - \omega^sq^{\langle \alpha_i' | \sigma^s(\alpha_j') \rangle}}$$

This function satisfies the following properties

\begin{equation}\label{0sym}
    (g_{ij,q}(t))^{-1} = g_{ij,q^{-1}}(t)
\end{equation}

\begin{equation}\label{1sym}
    g_{ij,q}(t) = g_{ji,q}(t)
\end{equation}

and,

\begin{equation}\label{2sym}
    g_{i\sigma^{\ell}(j),q^{-1}}(\omega^{\ell}t^{-1})=g_{ij,q}(t) \quad \quad \text{ for } \ell=0,1,\ldots, r-1 
\end{equation}

The first equation follows directly. For the second equation, we consider the cases $r=2$ and $r=3$ separately. If $r=2$ we have $\omega^{-s} = \omega^s$. Then

$$ g_{ij,q}(t) = \prod_{s\in \Z/r\Z} \frac{ q^{\langle \alpha_i' | \sigma^s(\alpha_j') \rangle}t-\omega^s}{ t - \omega^sq^{\langle \alpha_i' | \sigma^s(\alpha_j') \rangle}} = \prod_{s\in \Z/r\Z} \frac{ q^{\langle \sigma^{-s}(\alpha_i') | \alpha_j' \rangle}t-\omega^{-s}}{ t - \omega^{-s}q^{\langle \sigma^{-s}(\alpha_i') | \alpha_j' \rangle}} = \prod_{s\in \Z/r\Z} \frac{ q^{\langle \alpha_j'| \sigma^{-s}(\alpha_i') \rangle}t-\omega^{-s}}{ t - \omega^{-s}q^{\langle \alpha_j' | \sigma^{-s}(\alpha_i') \rangle}} = g_{ji,q}(t)$$

If $r=3$, let $i_o$ be the fixed node a the Dynkin diagram of type $D_4$. Then, for $j\neq i_o$ we have

$$ g_{i_oj,q}(t) = \prod_{s\in \Z/r\Z} \frac{ q^{\langle \alpha_{i_o}' | \sigma^s(\alpha_j') \rangle}t-\omega^s}{ t - \omega^sq^{\langle \alpha_{i_o}' | \sigma^s(\alpha_j') \rangle}}= \prod_{s\in \Z/r\Z} \frac{ q^{\langle \alpha_{j}' | \sigma^{-s}(\alpha_{i_o}') \rangle}t-\omega^s}{ t - \omega^sq^{\langle \alpha_{j}' | \sigma^{-s}(\alpha_{i_o}') \rangle}} = \prod_{s\in \Z/r\Z} \frac{ q^{\langle \alpha_{j}' | \sigma^{s}(\alpha_{i_o}') \rangle}t-\omega^s}{ t - \omega^sq^{\langle \alpha_{j}' | \sigma^{s}(\alpha_{i_o}') \rangle}} = g_{ji_o}(t)$$

Finally, for the third equation, we have

$$ g_{ij,q}(t) = \prod_{s\in \Z/r\Z} \frac{ q^{\langle \alpha_i' | \sigma^s(\alpha_j') \rangle}t-\omega^s}{t - \omega^sq^{\langle \alpha_i' | \sigma^s(\alpha_j') \rangle}} = \prod_{s\in \Z/r\Z} \frac{t  -\omega^sq^{-\langle \alpha_i' | \sigma^s(\alpha_j') \rangle}}{ q^{-\langle \alpha_i' | \sigma^s(\alpha_j') \rangle}t - \omega^s} $$

and thanks to $g_{i\sigma^\ell(j),q^{-1}}(\omega^{\ell}t^{-1}) = g_{\sigma^\ell(j)i,q^{-1}}(\omega^{\ell}t^{-1})$ for $\ell\in \Z$, we get

\begin{align*}  g_{\sigma^\ell(j)i,q^{-1}}(\omega^{\ell}t^{-1}) &= \prod_{s\in \Z/r\Z} \frac{\omega^{\ell}t^{-1}  -q^{\langle \sigma^{\ell}(\alpha_j') | \sigma^s(\alpha_{i}') \rangle}\omega^s}{\omega^{\ell} q^{\langle \sigma^{\ell}(\alpha_i') | \sigma^s(\alpha_{j}') \rangle}t^{-1} - \omega^s}  = \prod_{s\in \Z/r\Z} \frac{\omega^{\ell-s}  -q^{\langle \sigma^{\ell-s}(\alpha_j') | \alpha_{i}' \rangle}t}{\omega^{\ell-s} q^{\langle \sigma^{\ell-s}(\alpha_i') | \alpha_{j}' \rangle} - t} \\
&= \prod_{p\in \Z/r\Z} \frac{\omega^{p}-q^{\langle \alpha_i' |\sigma^{p}(\alpha_j') \rangle}t}{\omega^{p}q^{\langle \alpha_i' | \sigma^{p}(\alpha_j') \rangle} - t} = g_{ij,q}(t)
\end{align*}

Let $\mc{N}_q^{-}$ be the subalgebra of $U_q$ generated by $\gamma^{\pm 1/2}$ and $x_{i,k}^{-}$, for $k \in \Z$ and $i\in I_\sigma$. Any element in $\mc{N}_q^{-}$ is a sum of elements of the form $\gamma^{l/2}x_{j_i,k_1}^{-}x_{j_2,k_2}^{-}\cdots x_{j_r,k_m}^{-}$, and so it is part of a product $\gamma^{\ell/2}x_{j_1}^{-}(v_1)\cdots x_{j_m}^{-}(v_m)$. 

\begin{rem}
    If we just consider the case $r=2$, the algebra is well defined for any $i\in I_0$. But we restrict the index set $I$ to the set $I_\sigma$ because the right commutation relation in the case $r=3$ occurs there.  
\end{rem}

Now, for an arbitrary element $ x_{j_1}^{-}(v_1)\cdots x_{j_m}^{-}(v_m)\in \mc{N}_q^-$, thanks to the identities (\ref{conm-PhiX_gen}) and (\ref{conm-PsiX_gen}), we have

\begin{multline}\label{xpsigen}
    x_{j_1}^{-}(v_1)\cdots x_{j_{l-1}}^{-}(v_{l-1}) \psi_i(v_l\gamma^{1/2}\omega^{-s})x_{j_{l+1}}^{-}(v_{l+1}) \cdots  x_{j_m}^{-}(v_m)\\
    =  \prod_{k=1}^{l-1} g_{ij_k,q}(v_{k}v_l^{-1}\omega^s)^{-1} \psi_i(v_l\gamma^{1/2}\omega^{-s}) \overline{P}_l^{j_1, \ldots, j_m}
\end{multline}

and 

\begin{multline}\label{xphigen}
    x_{j_1}^{-}(v_1)\cdots x_{j_{l-1}}^{-}(v_{l-1}) \phi_i(u\gamma^{1/2})x_{j_{l+1}}^{-}(v_{l+1}) \cdots  x_{j_m}^{-}(v_m)\\
    = \prod_{k=1}^{l-1} g_{ij_k,q}(\gamma uv_{k}^{-1})^{-1} \phi_i(u\gamma^{1/2}) \overline{P}_l^{j_1, \ldots, j_m}
\end{multline}

where $\overline{P}_l^{j_1, \ldots, j_r} = x_{j_1}^{-}(v_1)\cdots x_{j_{l-1}}^{-}(v_{l-1}) x_{j_{l+1}}^{-}(v_{l+1}) \cdots  x_{j_m}^{-}(v_m)$. Then, for an element $x_i^+(u)$

\begin{multline*}
    [x_i^{+}(u), x_{j_1}^{-}(v_1)\cdots x_{j_m}^{-}(v_m)] = \sum_{l=1}^{m}  x_{j_1}^{-}(v_1)\cdots [x_i^{+}(u), x_{j_{l}}^{-}(v_l)] \cdots  x_{j_m}^{-}(v_m)\\
    = \sum_{l=1}^{m}  x_{j_1}^{-}(v_1)\cdots \Big(\sum_{s=0}^{r-1} \frac{\delta_{\sigma^s(i)j_l}}{q_i-q_i^{-1}}\Big(\delta\Big(\frac{u\omega^s}{v\gamma }\Big)\psi_i(v\gamma^{1/2}\omega^{-s}) - \delta\Big(\frac{u\gamma\omega^s }{v}\Big)\Big)\phi_i(u\gamma^{1/2})\Big) \cdots  x_{j_m}^{-}(v_m)\\
    = \frac{1}{q_i-q_i^{-1}} \Bigg\{\sum_{l=1}^{m} \sum_{s=0}^{r-1} \delta_{\sigma^s(i)j_l} x_{j_1}^{-}(v_1)\cdots x_{j_{l-1}}^{-}(v_{l-1}) \delta\Big(\frac{u\omega^s}{v_l\gamma }\Big)\psi_i(v_l\gamma^{1/2}\omega^{-s})x_{j_{l+1}}^{-}(v_{l+1}) \cdots  x_{j_m}^{-}(v_m)\\
    - \sum_{l=1}^{m} \sum_{s=0}^{r-1} \delta_{\sigma^s(i)j_l} x_{j_1}^{-}(v_1)\cdots x_{j_{l-1}}^{-}(v_{l-1}) \delta\Big(\frac{u\gamma \omega^s}{v_l}\Big)\phi_i(u\gamma^{1/2})x_{j_{l+1}}^{-}(v_{l+1}) \cdots  x_{j_m}^{-}(v_m) \Bigg\}\\
    = \frac{1}{q_i-q_i^{-1}}\Bigg\{\sum_{l=1}^{m} \sum_{s=0}^{r-1} \delta_{\sigma^s(i)j_l}\delta\Big(\frac{u \omega^s}{v_l\gamma }\Big)  x_{j_1}^{-}(v_1)\cdots x_{j_{l-1}}^{-}(v_{l-1}) \psi_i(v_l\gamma^{1/2}\omega^{-s})x_{j_{l+1}}^{-}(v_{l+1}) \cdots  x_{j_m}^{-}(v_m)\\
    - \sum_{l=1}^{m} \sum_{s=0}^{r-1} \delta_{\sigma^s(i)j_l}\delta\Big(\frac{u\gamma \omega^s}{v_l}\Big) x_{j_1}^{-}(v_1)\cdots x_{j_{l-1}}^{-}(v_{l-1}) \phi_i(u\gamma^{1/2})x_{j_{l+1}}^{-}(v_{l+1}) \cdots  x_{j_m}^{-}(v_m) \Bigg\}\\
    = \frac{1}{q_i-q_i^{-1}}\Bigg\{\sum_{l=1}^{m}  \sum_{s=0}^{r-1} \prod_{k=1}^{l-1}\delta_{\sigma^s(i)j_l}\delta\Big(\frac{u \omega^s}{v_l\gamma }\Big)   g_{ij_k,q}(v_{k}v_l^{-1}\omega^s)^{-1} \psi_i(v_l\gamma^{1/2}\omega^{-s}) \overline{P}_l^{j_1, \ldots, j_m}\\
    - \sum_{l=1}^{m} \sum_{s=0}^{r-1} \prod_{k=1}^{l-1}\delta_{\sigma^s(i)j_l}\delta\Big(\frac{u\gamma \omega^s}{v_l}\Big) g_{ij_k,q}(\gamma uv_{k}^{-1})^{-1} \phi_i(u\gamma^{1/2}) \overline{P}_l^{j_1, \ldots, j_m} \Bigg\}\\
    = \frac{\psi_i(u\gamma^{-1/2})}{q_i-q_i^{-1}}\sum_{l=1}^{m}  \sum_{s=0}^{r-1} \prod_{k=1}^{l-1}\delta_{\sigma^s(i)j_l}\delta\Big(\frac{u \omega^s}{v_l\gamma }\Big)   g_{ij_k,q}(v_{k}v_l^{-1}\omega^s)^{-1} \overline{P}_l^{j_1, \ldots, j_m}\\
    - \frac{\phi_i(u\gamma^{1/2}) }{q_i-q_i^{-1}} \sum_{l=1}^{m} \sum_{s=0}^{r-1} \prod_{k=1}^{l-1}\delta_{\sigma^s(i)j_l}\delta\Big(\frac{u\gamma \omega^s}{v_l}\Big) g_{ij_k,q}(v_lv_k^{-1}\omega^{-s})^{-1}\overline{P}_l^{j_1, \ldots, j_m}
\end{multline*}

The above commutation formula motivates the definition of the operators: 

$$ \Omega_{\psi_i}(u)(\overline{P}) = \sum_{l=1}^{m}  \sum_{s=0}^{r-1} \prod_{k=1}^{l-1}\delta_{\sigma^s(i)j_l}\delta\Big(\frac{u \omega^s}{v_l\gamma }\Big)   g_{ij_k,q}(v_{k}v_l^{-1}\omega^s)^{-1} \overline{P}_l^{j_1, \ldots, j_m}$$

$$\Omega_{\phi_i}(u)(\overline{P}) = \sum_{l=1}^{m} \sum_{s=0}^{r-1} \prod_{k=1}^{l-1}\delta_{\sigma^s(i)j_l}\delta\Big(\frac{u\gamma \omega^s}{v_l}\Big) g_{ij_k,q}(v_lv_k^{-1}\omega^{-s})^{-1}\overline{P}_l^{j_1, \ldots, j_m} $$

Their generating series defines operators $\Omega_{\psi_i}(a) : \mc{N}_q^- \to \mc{N}_q^-$ and $\Omega_{\phi_i}(a) : \mc{N}_q^- \to \mc{N}_q^- $ for every $a \in \Z$ as their coefficients; that is, $\Omega_{\psi_i}(u) = \sum_{a\in \Z} \Omega_{\psi_i}(a)u^{-a}$ and  $\Omega_{\phi_i}(u) = \sum_{a\in \Z} \Omega_{\phi_i}(a)u^{-a}$.

Then, we have proved:

\begin{lem}
    For any $P = \gamma^{\ell/2}x_{j_1}^{-}(v_1)\cdots x_{j_m}^{-}(v_m) \in \mc{N}_q^-$, let $\overline{P} = x_{j_1}^{-}(v_1)\cdots x_{j_m}^{-}(v_m) \in \mc{N}_q^-$. Then there exist formal series $\Omega_{\psi_i}(u)$ and $\Omega_{\phi_i}(u)$ such that 
    $$ [x_i^{+}(u), \overline{P}] = \frac{1}{q_i-q_i^{-1}} \left(\psi_i(u\gamma^{-1/2})\Omega_{\psi_i}(u)(\overline{P}) - \phi_i(u\gamma^{1/2})\Omega_{\phi_i}(u)(\overline{P})\right)$$
\end{lem}

Now, note that $\Omega_{\psi_i}(u)(1) = \Omega_{\phi_i}(u)(1) = 0$. Moreover, for 

$$ \overline{P} = x_{j_1}^{-}(v_1)\cdots x_{j_m}^{-}(v_m) = \sum_{n\in\Z}\sum_{\substack{k_1, \cdots k_m \in \Z \\ k_1 + \cdots k_m = n}} x_{j_1k_1}^{-}\cdots x_{j_mk_m}^{-}v_1^{-k_1}\cdots v_m^{-k_m}$$

we have 

\begin{align*}
    \psi_i(u\gamma^{-1/2})\Omega_{\psi_i}(u)(\overline{P}) &= \sum_{l\geq 0}\sum_{k\in\Z}\sum_{n\in\Z}\sum_{\substack{k_1, \cdots k_m \in \Z \\ k_1 + \cdots k_m = n}} \gamma^{-l/2}\psi_{i,l}\Omega_{\psi_i}(k)(x_{j_1k_1}^{-}\cdots x_{j_mk_m}^{-})v_1^{-k_1}\cdots v_m^{-k_m}u^{-k-l} \\
    &=\sum_{p\in\Z}\sum_{n\in\Z}\sum_{\substack{k_1, \cdots k_m \in \Z \\ k_1 + \cdots k_m = n}}\sum_{l\geq 0} \gamma^{-l/2}\psi_{i,l}\Omega_{\psi_i}(p-l)(x_{j_1k_1}^{-}\cdots x_{j_mk_m}^{-})v_1^{-k_1}\cdots v_m^{-k_m}u^{-p}
\end{align*}

but note also that 

$$ [x_i^{+}(u), \overline{P}] =  \sum_{p\in\Z}\sum_{n\in\Z}\sum_{\substack{k_1, \cdots k_m \in \Z \\ k_1 + \cdots k_m = n}}[x^{+}_{ip},x_{j_1k_1}^{-}\cdots x_{j_m k_m}^{-}]v_1^{-k_1}\cdots v_m^{-k_m}u^{-p}$$

is a finite sum. Hence, for a fixed tuple $(j_1,k_1), \ldots, (j_m,k_m)$ we should have that the sum 

$$ \sum_{l\geq 0} \gamma^{-l/2}\psi_{i,l}\Omega_{\psi_i}(p-l)(x_{j_1k_1}^{-}\cdots x_{j_mk_m}^{-})$$

must be finite. So, for fixed $p$ and tuple  $(j_1,k_1), \ldots, (j_m,k_m)$, we have 

$$ \Omega_{\psi_i}(p-l)(x_{j_1k_1}^{-}\cdots x_{j_mk_m}^{-}) = 0 $$

for $l$ sufficiently large. \\

\begin{prop}\label{Omega_x_commrealtions}
    Let $p,i\in I_\sigma$. Set $x_i^{-}(v) = \sum_{p\in \Z} x_{i,p}^{-}v^{-p}$, where $x_{i,p}^{-} : \mc{N}_q^- \to \mc{N}_q^-$ is considered as a left multiplication operator. Then
    \begin{align}
        \Omega_{\psi_p}(u)x_i^{-}(v) =& \Big(\sum_{s=0}^{r-1}\delta_{\sigma^s(p),i}\delta\Big(\frac{u\omega^s}{v\gamma}\Big)\Big) + g_{pi,q^{-1}}(v\gamma u^{-1})x_i^{-}(v)\Omega_{\psi_p}(u)  \label{eq_comm01_gen} \\
        \Omega_{\phi_p}(u)x_i^{-}(v) =& \Big(\sum_{s=0}^{r-1}\delta_{\sigma^s(p),i}\delta\Big(\frac{u\gamma \omega^s}{v}\Big)\Big) + g_{pi,q}(u\gamma v^{-1})x_i^{-}(v)\Omega_{\phi_p}(u)\label{eq_comm02_gen}
    \end{align}
\end{prop}

\dem Let $\overline{P} = x_{j_1}^{-}(v_1)\cdots x_{j_m}^{-}(v_m)$, then 

\begin{align*}
    \Omega_{\psi_p}(u)x_i^{-}(v)(\overline{P}) &= \Omega_{\psi_p}(u)(x_i^{-}(v)x_{j_1}^{-}(v_1)\cdots x_{j_m}^{-}(v_m))\\
    &= \sum_{l=0}^{m}  \sum_{s=0}^{r-1} \prod_{k=0}^{l-1}\delta_{\sigma^s(p)j_l}\delta\Big(\frac{u \omega^s}{v_l\gamma }\Big)   g_{pj_k,q}(v_{k}v_l^{-1}\omega^s)^{-1} \overline{P}_l^{j_0,j_1, \ldots, j_m}
\end{align*}

where $j_0=i$. Then,

{\small
\begin{align*}
    \Omega_{\psi_p}(u)x_i^{-}(v)(\overline{P}) &= \Big(\sum_{s=0}^{r-1}\delta_{\sigma^s(p),i}\delta\Big(\frac{u\omega^s}{v\gamma}\Big)\Big)\overline{P} + x_i^{-}(v)\sum_{l=1}^{m}  \sum_{s=0}^{r-1} \prod_{k=0}^{l-1}\delta_{\sigma^s(p)j_l}\delta\Big(\frac{u \omega^s}{v_l\gamma }\Big)   g_{pj_k,q^{-1}}(v_{k}v_l^{-1}\omega^s) \overline{P}_l^{j_1, \ldots, j_m}\\
    &= \Big(\sum_{s=0}^{r-1}\delta_{\sigma^s(p),i}\delta\Big(\frac{u\omega^s}{v\gamma}\Big)\Big)\overline{P}\\
    &+ x_i^{-}(v)g_{pi,q^{-1}}(v\gamma u^{-1})\sum_{l=1}^{m}  \sum_{s=0}^{r-1} \prod_{k=1}^{l-1}\delta_{\sigma^s(i)j_l}\delta\Big(\frac{u \omega^s}{v_l\gamma }\Big)   g_{ij_k,q}(v_{k}v_l^{-1}\omega^s)^{-1} \overline{P}_l^{j_1, \ldots, j_m}\\
    &=\Big(\sum_{s=0}^{r-1}\delta_{\sigma^s(p),i}\delta\Big(\frac{u\omega^s}{v\gamma}\Big)\Big)\overline{P} + x_i^{-}(v)g_{pi,q^{-1}}(v\gamma u^{-1})\Omega_{\psi_p}(u)(\overline{P})
\end{align*}}

Similarly, 

{\small
\begin{align*}
    \Omega_{\phi_p}(u)x_i^{-}(v)(\overline{P}) &= \Omega_{\phi_p}(u)(x_i^{-}(v)x_{j_1}^{-}(v_1)\cdots x_{j_m}^{-}(v_m))\\
    &= \sum_{l=0}^{m} \sum_{s=0}^{r-1} \prod_{k=0}^{l-1}\delta_{\sigma^s(p)j_l}\delta\Big(\frac{u\gamma \omega^s}{v_l}\Big) g_{pj_k,q}(v_lv_k^{-1}\omega^{-s})\overline{P}_l^{j_0,j_1, \ldots, j_m} \\
    &= \Big(\sum_{s=0}^{r-1}\delta_{\sigma^s(p),i}\delta\Big(\frac{u\gamma \omega^s}{v}\Big)\Big)\overline{P} + x_i^{-}(v)\sum_{l=1}^{m} \sum_{s=0}^{r-1} \prod_{k=0}^{l-1}\delta_{\sigma^s(p)j_l}\delta\Big(\frac{u\gamma \omega^s}{v_l}\Big) g_{pj_k,q}(v_lv_k^{-1}\omega^{-s})\overline{P}_l^{j_1, \ldots, j_m}\\
    &= \Big(\sum_{s=0}^{r-1}\delta_{\sigma^s(p),i}\delta\Big(\frac{u\gamma \omega^s}{v}\Big)\Big)\overline{P}\\
    &+ x_i^{-}(v)g_{pi,q}(u\gamma v^{-1})\sum_{l=1}^{m} \sum_{s=0}^{r-1} \prod_{k=1}^{l-1}\delta_{\sigma^s(p)j_l}\delta\Big(\frac{u\gamma \omega^s}{v_l}\Big) g_{pj_k,q}(v_lv_k^{-1}\omega^{-s})\overline{P}_l^{j_1, \ldots, j_m}\\
    &= \Big(\sum_{s=0}^{r-1}\delta_{\sigma^s(p),i}\delta\Big(\frac{u\gamma \omega^s}{v}\Big)\Big)\overline{P} + x_i^{-}(v)g_{pi,q}(u\gamma v^{-1})\Omega_{\phi_p}(u)(\overline{P})
\end{align*}}

as we want. 

\findem

\begin{prop}\label{Omega_commrealtions}
    For $j,k\in I_\sigma$ we have the following commutation relations
    \begin{align}
        g_{jk,q}\left(\frac{u_1}{u_2}\right)\Omega_{\psi_j}(u_1)\Omega_{\psi_k}(u_2) =& \Omega_{\psi_k}(u_2)\Omega_{\psi_j}(u_1) \label{eq_comm03_gen}\\
        g_{jk,q}\left(\frac{u_1}{u_2}\right)\Omega_{\phi_j}(u_1)\Omega_{\phi_k}(u_2) =& \Omega_{\phi_k}(u_2)\Omega_{\phi_j}(u_1)\\
        g_{jk,q}\left(\frac{u_1\gamma^2}{u_2}\right)\Omega_{\phi_j}(u_1)\Omega_{\psi_k}(u_2) =& \Omega_{\psi_k}(u_2)\Omega_{\phi_j}(u_1)
    \end{align}

\end{prop}

\dem We prove the first relation. The second and the third ones have similar proofs. Let us consider 

\begin{align*}\Omega_{\psi_j}(u_1)\Omega_{\psi_k}(u_2)x_i^{-}(v) &= \sum_{s=0}^{r-1}\delta_{\sigma^s(k),i}\delta\left(\frac{u_2\omega^s}{v\gamma} \right)\Omega_{\psi_j}(u_1) + \sum_{s=0}^{r-1} \delta_{\sigma^s(j),i}\delta\left(\frac{u_1\omega^{s}}{v\gamma}\right)g_{ki,q^{-1}}\left(\frac{v\gamma}{u_2}\right)\Omega_{\psi_k}(u_2)\\
&+ g_{ki,q^{-1}}\left(\frac{v\gamma }{u_2}\right)g_{ji,q^{-1}}\left(\frac{v\gamma}{u_1}\right)x_i^{-}(v)\Omega_{\psi_j}(u_1)\Omega_{\psi_k}(u_2)
\end{align*}

and 

\begin{align*}\Omega_{\psi_k}(u_2)\Omega_{\psi_j}(u_1)x_i^{-}(v) &= \sum_{s=0}^{r-1}\delta_{\sigma^s(j),i}\delta\left(\frac{u_1\omega^s}{v\gamma} \right)\Omega_{\psi_k}(u_2) + \sum_{s=0}^{r-1} \delta_{\sigma^s(k),i}\delta\left(\frac{u_2\omega^{s}}{v\gamma}\right)g_{ji,q^{-1}}\left(\frac{v\gamma}{u_1}\right)\Omega_{\psi_j}(u_1)\\
&+ g_{ji,q^{-1}}\left(\frac{v\gamma }{u_1}\right)g_{ki,q^{-1}}\left(\frac{v\gamma}{u_2}\right)x_i^{-}(v)\Omega_{\psi_k}(u_2)\Omega_{\psi_j}(u_1)
\end{align*}

Set $S = g_{jk,q}\left(\frac{u_1}{u_2}\right)\Omega_{\psi_j}(u_1)\Omega_{\psi_k}(u_2) - \Omega_{\psi_k}(u_2)\Omega_{\psi_j}(u_1)$. Combining like terms, we get the following expression

\begin{align*}
    S &= \sum_{s=0}^{r-1}\delta_{\sigma^{s}(k),i}\delta\left(\frac{u_2\omega^s}{v\gamma}\right)\Omega_{\psi_j}\Bigg\{ g_{jk,q}\left(\frac{u_1}{u_2}\right) - g_{ji,q^{-1}}\left(\frac{v\gamma}{u_1}\right)\Bigg\}\\
    &+ \sum_{s=0}^{r-1}\delta_{\sigma^{s}(j),i}\delta\left(\frac{u_1\omega^s}{v\gamma}\right)\Omega_{\psi_k}\Bigg\{ g_{jk,q}\left(\frac{u_1}{u_2}\right)g_{ki,q^{-1}}\left(\frac{v\gamma}{u_2}\right) - 1\Bigg\}\\
    &+ g_{ki,q^{-1}}\left(\frac{v\gamma }{u_2}\right)g_{ji,q^{-1}}\left(\frac{v\gamma}{u_1}\right)x_i^{-}(v)S
\end{align*}

Thanks to identities (\ref{0sym}), (\ref{1sym}), and (\ref{2sym}), the first and second terms vanish. Hence, for a monomial $x_{i_1}^{-}(v_1)\cdots x_{i_m}^{-}(v_m)$, we have 

$$ Sx_{i_1}^{-}(v_1)\cdots x_{i_m}^{-}(v_m) = \Bigg\{\prod_{\ell=1}^{m} g_{ki_\ell,q^{-1}}\left(\frac{v_{\ell}\gamma }{u_2}\right)g_{ji_{\ell},q^{-1}}\left(\frac{v_{\ell}\gamma }{u_1}\right)\Bigg\} x_{i_1}^{-}(v_1)\cdots x_{i_m}^{-}(v_m)S$$

If we apply this formula to 1, and notice that $S1=0$, we get $S=0$, and the desired formula follows.
    
\findem

\begin{rem}
    In case $r=1$ or 2, we can change the index set $I_\sigma$ for $I_0$, and the above two propositions are still true.
\end{rem}

The formulas in the previous Proposition can also be written as

\begin{align}
        \prod_{s=0}^{r-1}(q^{\langle \alpha_i' | \sigma^s(\alpha_j') \rangle}u_1-\omega^su_2)\Omega_{\psi_i}(u_1)\Omega_{\psi_j}(u_2) =& \prod_{s=0}^{r-1}(u_1 - \omega^sq^{\langle \alpha_i' | \sigma^s(\alpha_j')\rangle} u_2)\Omega_{\psi_j}(u_2)\Omega_{\psi_i}(u_1) \label{eq_comm033_gen}\\
        \prod_{s=0}^{r-1}(q^{\langle \alpha_i' | \sigma^s(\alpha_j') \rangle}u_1-\omega^su_2)\Omega_{\phi_i}(u_1)\Omega_{\phi_j}(u_2) =& \prod_{s=0}^{r-1}(u_1 - \omega^sq^{\langle \alpha_i' | \sigma^s(\alpha_j')\rangle} u_2)\Omega_{\phi_j}(u_2)\Omega_{\phi_i}(u_1)\\
        \prod_{s=0}^{r-1}(q^{\langle \alpha_i' | \sigma^s(\alpha_j') \rangle}u_1\gamma^2-\omega^su_2)\Omega_{\phi_i}(u_1)\Omega_{\psi_j}(u_2) =& \prod_{s=0}^{r-1}(u_1\gamma^2 - \omega^sq^{\langle \alpha_i' | \sigma^s(\alpha_j')\rangle} u_2)\Omega_{\psi_j}(u_2)\Omega_{\phi_i}(u_1)
    \end{align}

Moreover, equation (\ref{eq_comm01_gen}) can be written as

\begin{multline}\label{eq_comm011_gen}
    \prod_{s=0}^{r-1}(q^{\langle \alpha_i' | \sigma^s(\alpha_j') \rangle}v\gamma - \omega^s u) \Omega_{\psi_i}(u)x_j^{-}(v) = \prod_{s=0}^{r-1}(q^{\langle \alpha_i' | \sigma^s(\alpha_j') \rangle}v\gamma - \omega^s u)\Big(\sum_{s=0}^{r-1}\delta_{\sigma^s(i),j}\delta\Big(\frac{u\omega^s}{v\gamma}\Big)\Big) \\
    + \prod_{s=0}^{r-1}(v\gamma - \omega^s u q^{\langle \alpha_i' | \sigma^s(\alpha_j') \rangle})x_j^{-}(v)\Omega_{\psi_i}(u)
\end{multline}

To define the Kashiwara algebra, we need a component-wise version of the equations (\ref{comx_ix_j_currentform}), (\ref{eq_comm033_gen}) and (\ref{eq_comm011_gen}). We split the resulting formulas in the two cases $r=2$ and $r=3$. For simplicity, we write $a_{i\sigma^s(j)} := \langle \alpha_i' | \sigma^s(\alpha_j')\rangle$\\

Let us first consider when $r=2$. In this case, equation (\ref{comx_ix_j_currentform}) becomes

\begin{multline}\label{Kas_commxx_r2}
    x_{i,k+2}^-x_{j,s}^- - (q^{a_{ij}} - q^{a_{i\sigma(j)}})x_{i,k+1}^-x_{j,s+1}^- - q^{a_{ij} + a_{i\sigma(j)}}x_{i,k}^-x_{j,s+2}^- \\
    = q^{a_{ij} + a_{i\sigma(j)}}x_{j,s}^-x_{i,k+2}^- + (q^{a_{ij}} - q^{a_{i\sigma(j)}})x_{j,s+1}^-x_{i,k+1}^- - x_{j,s+2}^-x_{i,k}^-
\end{multline}

Equation (\ref{eq_comm033_gen}) becomes

\begin{multline}\label{Kas_commOO_r2}
    q^{a_{ij} + a_{i\sigma(j)}}\Omega_{\psi_i}(k+2)\Omega_{\psi_j}(s) + (q^{a_{ij}} - q^{a_{i\sigma(j)}})\Omega_{\psi_i}(k+1)\Omega_{\psi_j}(s+1) - \Omega_{\psi_i}(k)\Omega_{\psi_j}(s+2)\\
= \Omega_{\psi_j}(s)\Omega_{\psi_i}(k+2) - (q^{a_{ij}} - q^{a_{i\sigma(j)}})\Omega_{\psi_j}(s+1)\Omega_{\psi_i}(k+1) - q^{a_{ij} + a_{i\sigma(j)}}\Omega_{\psi_j}(s+2)\Omega_{\psi_i}(k)
\end{multline}

and (\ref{eq_comm011_gen}) may be written out as follows

\begin{multline}\label{Kas_commxO_r2}
    q^{a_{ij} + a_{i\sigma(j)}}\gamma^2\Omega_{\psi_i}(k)x_{j,s+2}^- + (q^{a_{ij}} - q^{a_{i\sigma(j)}})\gamma\Omega_{\psi_i}(k+1)x_{j,s+1}^- - \Omega_{\psi_i}(k+2)x_{j,s}^-\\
= (q^{a_{ij} + a_{i\sigma(j)}} + q^{a_{ij}} - q^{a_{i\sigma(j)}} - 1)\gamma^{k+2}\delta_{k,-s-2}(\delta_{i,j} + (-1)^{k+2}\delta_{\sigma(i),j})  \\
+\gamma^2x_{j,s+2}^-\Omega_{\psi_i}(k) - (q^{a_{ij}} - q^{a_{i\sigma(j)}})\gamma x_{j,s+1}^-\Omega_{\psi_i}(k+1) - q^{a_{ij} + a_{i\sigma(j)}}x_{j,s}^-\Omega_{\psi_i}(k+2)
\end{multline}

In the case when $r=3$, equation (\ref{comx_ix_j_currentform}) turns out to be

\begin{multline}\label{Kas_commxx_r3}
    x_{i,k+3}^-x_{j,s}^- - (\omega^2q^{a_{i\sigma^2(j)}} + \omega q^{a_{i\sigma(j)}} + q^{a_{ij}})x_{i,k+2}^-x_{j,s+1}^-\\
    +  (\omega^2q^{a_{ij} + a_{i\sigma^2(j)}} + \omega q^{a_{ij} + a_{i\sigma(j)}} + q^{a_{i\sigma(j)} + a_{i\sigma^2(j)}})x_{i,k+1}^-x_{j,s+2}^- - q^{a_{ij} + a_{i\sigma(j)} + a_{i\sigma^2(j)}}x_{i,k}^-x_{j,s+3}^-\\
    = q^{a_{ij} + a_{i\sigma(j)} + a_{i\sigma^2(j)}}x_{j,s}^-x_{i,k+3}^- - (\omega^2q^{a_{ij} + a_{i\sigma(j)}} + \omega q^{a_{ij} + a_{i\sigma^2(j)}} + q^{a_{i\sigma(j)} + a_{i\sigma^2(j)}})x_{j,s+1}^-x_{i,k+2}^-\\
    + (\omega^2q^{a_{i\sigma(j)}} + \omega q^{a_{i\sigma^2(j)}} + q^{a_{ij}})x_{j,s+2}^-x_{i,k+1}^- -x_{j,s+3}^-x_{i,k}^-
\end{multline}

and the equation (\ref{eq_comm033_gen}) can be written as

\begin{multline}\label{Kas_commOO_r3}
    q^{a_{ij} + a_{i\sigma(j)} + a_{i\sigma^2(j)}}\Omega_{\psi_i}(k+3)\Omega_{\psi_j}(s) - (\omega^2q^{a_{ij} + a_{i\sigma(j)}} + \omega q^{a_{ij} + a_{i\sigma^2(j)}} + q^{a_{i\sigma(j)} + a_{i\sigma^2(j)}})\Omega_{\psi_i}(k+2)\Omega_{\psi_j}(s+1)\\
    + (\omega^2q^{a_{i\sigma(j)}} + \omega q^{a_{i\sigma^2(j)}} + q^{a_{ij}})\Omega_{\psi_i}(k+1)\Omega_{\psi_j}(s+2) - \Omega_{\psi_i}(k)\Omega_{\psi_j}(s+3)\\
    = \Omega_{\psi_j}(s)\Omega_{\psi_i}(k+3) - (\omega^2q^{a_{i\sigma^2(j)}} + \omega q^{a_{i\sigma(j)}} + q^{a_{ij}})\Omega_{\psi_j}(s+1)\Omega_{\psi_i}(k+2)\\ 
    +  (\omega^2q^{a_{ij} + a_{i\sigma^2(j)}} + \omega q^{a_{ij} + a_{i\sigma(j)}} + q^{a_{i\sigma(j)} + a_{i\sigma^2(j)}})\Omega_{\psi_j}(s+2)\Omega_{\psi_i}(k+1) - q^{a_{ij} + a_{i\sigma(j)} + a_{i\sigma^2(j)}}\Omega_{\psi_j}(s+3)\Omega_{\psi_i}(k)
\end{multline}

finally, equation (\ref{eq_comm011_gen}) in components becomes

\begin{multline}\label{Kas_commxO_r3}
    q^{a_{ij} + a_{i\sigma(j)} + a_{i\sigma^2(j)}}\gamma^3\Omega_{\psi_i}(k)x_{j,s+3}^- - (\omega^2q^{a_{ij} + a_{i\sigma(j)}} + \omega q^{a_{ij} + a_{i\sigma^2(j)}} + q^{a_{i\sigma(j)} + a_{i\sigma^2(j)}})\gamma^2\Omega_{\psi_i}(k+1)x_{j,s+2}^-\\
   + (\omega^2q^{a_{i\sigma(j)}} + \omega q^{a_{i\sigma^2(j)}} + q^{a_{ij}})\gamma\Omega_{\psi_i}(k+2)x_{j,s+1}^- - \Omega_{\psi_i}(k+3)x_{j,s}^-\\
   = \bigl(q^{a_{ij} + a_{i\sigma(j)} + a_{i\sigma^2(j)}} - \omega^2q^{a_{ij} - a_{i\sigma(j)}} - \omega q^{a_{ij} + a_{i\sigma^2(j)}} - q^{a_{i\sigma(j)} + a_{i\sigma^2(j)}}\\
   + \omega^2q^{a_{i\sigma(j)}} + \omega q^{a_{i\sigma^2(j)}} + q^{a_{ij}} - 1 \bigl)\gamma^{k+3}\delta_{s,-k-3}(\delta_{ij} + \omega^{-k-3}\delta_{\sigma(i),j} + \omega^{-2(k+3)}\delta_{\sigma^2(i),j})\\
   + \gamma^3 x_{j,s+3}^-\Omega_{\psi_i}(k) - (\omega^2q^{a_{i\sigma^2(j)}} + \omega q^{a_{i\sigma(j)}} + q^{a_{ij}})\gamma^2x_{j,s+2}^-\Omega_{\psi_i}(k+1)\\ 
    +  (\omega^2q^{a_{ij} + a_{i\sigma^2(j)}} + \omega q^{a_{ij} + a_{i\sigma(j)}} + q^{a_{i\sigma(j)} + a_{i\sigma^2(j)}})\gamma x_{j,s+1}^-\Omega_{\psi_i}(k+2) - q^{a_{ij} + a_{i\sigma(j)} + a_{i\sigma^2(j)}}x_{j,s}^-\Omega_{\psi_i}(k+3)
\end{multline}

We can also write (\ref{eq_comm01_gen}) in components and as an operator in $\mc{N}_q^-$ as follows

\begin{equation}\label{action_omega_on_x}
    \Omega_{\psi_i}(k)x_{j,m}^- = \sum_{s=0}^{r-1}\delta_{\sigma^s(i),j}\delta_{m,-k}\omega^{-sk}\gamma^k + \sum_{s\geq 0}g_{ij,q^{-1}}(s)x_{j,m+k}^-\Omega_{\psi_i}(k-s)\gamma^s
\end{equation}

where $g_{ij,q^{-1}}(s)$, $s\in \Z$, are the coefficients of the function $g_{ij,q^{-1}}(t) = \sum_{s\geq 0} g_{ij,q^{-1}}(s)t^s$

\section{The Kashiwara algebra}

In this section, we are going to define the Kashiwara algebra associated with the natural partition of a twisted quantum affine algebra and deduce its main properties, as was done in \cite{CFM01} for the untwisted cases of ADE type (see also \cite{AFM} for a remark on the remaining untwisted cases).

\begin{defi}
    The Kashiwara algebra, denoted by $\mc{K}_q$, is defined to be the $\C(q^{1/2})$-algebra with generators $\Omega_{\psi_j}(k)$, $x_{i,p}^{-}$, and $\gamma^{\pm 1/2}$, where $k,p\in \Z$, $i,j\in I_\sigma$ and which are subject to $\gamma^{\pm 1/2}\gamma^{\mp 1/2} = 1$ and the relations (\ref{Kas_commxx_r2}), (\ref{Kas_commOO_r2}) and (\ref{Kas_commxO_r2}) when $r=2$, and (\ref{Kas_commxx_r3}), (\ref{Kas_commOO_r3}) and (\ref{Kas_commxO_r3}) when $r=3$.  
\end{defi}

We are going to define an inner product on $\mc{N}_q^-$, to do so, we need the following two lemmas. \\

Let $\overline{\alpha} : \mc{K}_q \to \mc{K}_q$ be the map defined in generators by

$$ \overline{\alpha}(\gamma^{\pm 1/2}) = \gamma^{\pm 1/2}, \quad \quad \overline{\alpha}(x^{-}_{i,p}) = \Omega_{\psi_i}(-p), \quad\quad \overline{\alpha}(\Omega_{\psi_i}(k)) = x^{-}_{i,-k}  $$

for all $k,p\in \Z$, it is extended by $\C(q^{1/2})$-linearity.

\begin{lem}
    The $\C(q^{1/2})$-linear map $\overline{\alpha} : \mc{K}_q \to \mc{K}_q$ is an involutive anti-automorphism. 
\end{lem}

\dem
    The map $\overline{\alpha}$ is clearly an involution and a bijection. So, it is enough to check that it is an anti-homomorphism. To do this, we check that the relations are preserved. When $r=2$, we get

    \begin{multline*}
        \overline{\alpha}(x_{i,k+2}^-x_{j,s}^- - (q^{a_{ij}} - q^{a_{i\sigma(j)}})x_{i,k+1}^-x_{j,s+1}^- - q^{a_{ij} + a_{i\sigma(j)}}x_{i,k}^-x_{j,s+2}^-) \\
        = \Omega_{\psi_j}(-s)\Omega_{\psi_i}(-k-2) - (q^{a_{ij}} - q^{a_{i\sigma(j)}})\Omega_{\psi_j}(-s-1)\Omega_{\psi_i}(-k-1) - q^{a_{ij} + a_{i\sigma(j)}}\Omega_{\psi_j}(-s-2)\Omega_{\psi_i}(-k)\\
        =q^{a_{ij} + a_{i\sigma(j)}}\Omega_{\psi_i}(-k-2)\Omega_{\psi_j}(-s) + (q^{a_{ij}} - q^{a_{i\sigma(j)}})\Omega_{\psi_i}(-k-1)\Omega_{\psi_j}(-s-1) - \Omega_{\psi_i}(-k)\Omega_{\psi_j}(-s-2) \\
        = \overline{\alpha}(q^{a_{ij} + a_{i\sigma(j)}}x_{j,s}^-x_{i,k+2}^- + (q^{a_{ij}} - q^{a_{i\sigma(j)}})x_{j,s+1}^-x_{i,k+1}^- - x_{j,s+2}^-x_{i,k}^-)
    \end{multline*}

    and 

    \begin{multline*}
        \overline{\alpha}(q^{a_{ij} + a_{i\sigma(j)}}\gamma^2\Omega_{\psi_i}(k)x_{j,s+2}^- + (q^{a_{ij}} - q^{a_{i\sigma(j)}})\gamma\Omega_{\psi_i}(k+1)x_{j,s+1}^- - \Omega_{\psi_i}(k+2)x_{j,s}^-) \\
        = q^{a_{ij} + a_{i\sigma(j)}}\gamma^2\Omega_{\psi_j}(-s-2)x_{i,-k}^- + (q^{a_{ij}} - q^{a_{i\sigma(j)}})\gamma\Omega_{\psi_j}(-s-1)x_{i,-k-1}^- - \Omega_{\psi_j}(-s)x_{i,-k-2}^-\\
        =(q^{a_{ij} + a_{i\sigma(j)}} + q^{a_{ij}} - q^{a_{i\sigma(j)}} - 1)\gamma^{k+2}\delta_{k,-s-2}(\delta_{i,j} + (-1)^{k+2}\delta_{\sigma(i),j})  \\
        +\gamma^2x_{i,-k}^-\Omega_{\psi_j}(-s-2) - (q^{a_{ij}} - q^{a_{i\sigma(j)}})\gamma x_{i,-k-1}^-\Omega_{\psi_j}(-s-1) - q^{a_{ij} + a_{i\sigma(j)}}x_{i,-k-2}^-\Omega_{\psi_j}(-s)\\
        =\overline{\alpha}\left( ((q^{a_{ij} + a_{i\sigma(j)}} + q^{a_{ij}} - q^{a_{i\sigma(j)}} - 1)\gamma^{k+2}\delta_{k,-s-2}(\delta_{i,j} + (-1)^{k+2}\delta_{\sigma(i),j}) \right.\\
        \left. +\gamma^2x_{j,s+2}^-\Omega_{\psi_i}(k) - (q^{a_{ij}} - q^{a_{i\sigma(j)}})\gamma x_{j,s+1}^-\Omega_{\psi_i}(k+1) - q^{a_{ij} + a_{i\sigma(j)}}x_{j,s}^-\Omega_{\psi_i}(k+2) \right)
    \end{multline*}

    and for $r=3$ we get

    \begin{multline*}
        \overline{\alpha}(x_{i,k+3}^-x_{j,s}^- - (\omega^2q^{a_{i\sigma^2(j)}} + \omega q^{a_{i\sigma(j)}} + q^{a_{ij}})x_{i,k+2}^-x_{j,s+1}^-\\
    +  (\omega^2q^{a_{ij} + a_{i\sigma^2(j)}} + \omega q^{a_{ij} + a_{i\sigma(j)}} + q^{a_{i\sigma(j)} + a_{i\sigma^2(j)}})x_{i,k+1}^-x_{j,s+2}^- - q^{a_{ij} + a_{i\sigma(j)} + a_{i\sigma^2(j)}}x_{i,k}^-x_{j,s+3}^- )\\
        = \Omega_{\psi_j}(-s)\Omega_{\psi_i}(-k-3) - (\omega^2q^{a_{i\sigma^2(j)}} + \omega q^{a_{i\sigma(j)}} + q^{a_{ij}})\Omega_{\psi_j}(-s-1)\Omega_{\psi_i}(-k-2)\\
        +  (\omega^2q^{a_{ij} + a_{i\sigma^2(j)}} + \omega q^{a_{ij} + a_{i\sigma(j)}} + q^{a_{i\sigma(j)} + a_{i\sigma^2(j)}})\Omega_{\psi_j}(-s-2)\Omega_{\psi_i}(-k-1) \\
        - q^{a_{ij} + a_{i\sigma(j)} + a_{i\sigma^2(j)}}\Omega_{\psi_j}(-s-3)\Omega_{\psi_i}(-k)\\
        =q^{a_{ij} + a_{i\sigma(j)} + a_{i\sigma^2(j)}}\Omega_{\psi_i}(-k-3)\Omega_{\psi_j}(-s) \\
        - (\omega^2q^{a_{ij} + a_{i\sigma(j)}} + \omega q^{a_{ij} + a_{i\sigma^2(j)}} + q^{a_{i\sigma(j)} + a_{i\sigma^2(j)}})\Omega_{\psi_i}(-k-2)\Omega_{\psi_j}(-s-1) \\
        + (\omega^2q^{a_{i\sigma(j)}} + \omega q^{a_{i\sigma^2(j)}} + q^{a_{ij}})\Omega_{\psi_i}(-k-1)\Omega_{\psi_j}(-s-2) - \Omega_{\psi_i}(-k)\Omega_{\psi_j}(-s-3)\\
        = \overline{\alpha}(q^{a_{ij} + a_{i\sigma(j)} + a_{i\sigma^2(j)}}x_{j,s}^-x_{i,k+3}^- - (\omega^2q^{a_{ij} + a_{i\sigma(j)}} + \omega q^{a_{ij} + a_{i\sigma^2(j)}} + q^{a_{i\sigma(j)} + a_{i\sigma^2(j)}})x_{j,s+1}^-x_{i,k+2}^-\\
        + (\omega^2q^{a_{i\sigma(j)}} + \omega q^{a_{i\sigma^2(j)}} + q^{a_{ij}})x_{j,s+2}^-x_{i,k+1}^- -x_{j,s+3}^-x_{i,k}^-)
    \end{multline*}

    and 

    \begin{multline*}
        \overline{\alpha}(q^{a_{ij} + a_{i\sigma(j)} + a_{i\sigma^2(j)}}\gamma^3\Omega_{\psi_i}(k)x_{j,s+3}^- - (\omega^2q^{a_{ij} + a_{i\sigma(j)}} + \omega q^{a_{ij} + a_{i\sigma^2(j)}} + q^{a_{i\sigma(j)} + a_{i\sigma^2(j)}})\gamma^2\Omega_{\psi_i}(k+1)x_{j,s+2}^- \\
        + (\omega^2q^{a_{i\sigma(j)}} + \omega q^{a_{i\sigma^2(j)}} + q^{a_{ij}})\gamma\Omega_{\psi_i}(k+2)x_{j,s+1}^- - \Omega_{\psi_i}(k+3)x_{j,s}^-)\\
        = q^{a_{ij} + a_{i\sigma(j)} + a_{i\sigma^2(j)}}\gamma^3\Omega_{\psi_j}(-s-3)x_{i,-k}^- - (\omega^2q^{a_{ij} + a_{i\sigma(j)}} + \omega q^{a_{ij} + a_{i\sigma^2(j)}} + q^{a_{i\sigma(j)} + a_{i\sigma^2(j)}})\gamma^2\Omega_{\psi_j}(-s-2)x_{i,-k-1}^-\\
        + (\omega^2q^{a_{i\sigma(j)}} + \omega q^{a_{i\sigma^2(j)}} + q^{a_{ij}})\gamma\Omega_{\psi_j}(-s-1)x_{i,-k-2}^- - \Omega_{\psi_j}(-s)x_{i,-k-3}^-)\\
        = \bigl(q^{a_{ij} + a_{i\sigma(j)} + a_{i\sigma^2(j)}} - \omega^2q^{a_{ij} - a_{i\sigma(j)}} - \omega q^{a_{ij} + a_{i\sigma^2(j)}} - q^{a_{i\sigma(j)} + a_{i\sigma^2(j)}}\\
   + \omega^2q^{a_{i\sigma(j)}} + \omega q^{a_{i\sigma^2(j)}} + q^{a_{ij}} - 1 \bigl)\gamma^{k+3}\delta_{s,-k-3}(\delta_{ij} + \omega^{-p-3}\delta_{\sigma(i),j} + \omega^{-2(k+3)}\delta_{\sigma^2(i),j})\\
   + \gamma^3 x_{i,-k}^-\Omega_{\psi_j}(-s-3) - (\omega^2q^{a_{i\sigma^2(j)}} + \omega q^{a_{i\sigma(j)}} + q^{a_{ij}})\gamma^2x_{i,-k-1}^-\Omega_{\psi_j}(-s-2)\\ 
    +  (\omega^2q^{a_{ij} + a_{i\sigma^2(j)}} + \omega q^{a_{ij} + a_{i\sigma(j)}} + q^{a_{i\sigma(j)} + a_{i\sigma^2(j)}})\gamma x_{i,-k-2}^-\Omega_{\psi_j}(-s-1) - q^{a_{ij} + a_{i\sigma(j)} + a_{i\sigma^2(j)}}x_{i,-k-3}^-\Omega_{\psi_j}(-s)\\
        =\overline{\alpha}\Bigl( \bigl(q^{a_{ij} + a_{i\sigma(j)} + a_{i\sigma^2(j)}} - \omega^2q^{a_{ij} - a_{i\sigma(j)}} - \omega q^{a_{ij} + a_{i\sigma^2(j)}} - q^{a_{i\sigma(j)} + a_{i\sigma^2(j)}} \\
        + \omega^2q^{a_{i\sigma(j)}} + \omega q^{a_{i\sigma^2(j)}} + q^{a_{ij}} - 1 \bigl)\gamma^{k+3}\delta_{s,-k-3}(\delta_{ij} + \omega^{-k-3}\delta_{\sigma(i),j} + \omega^{-2(k+3)}\delta_{\sigma^2(i),j})\\
        + \gamma^3 x_{j,s+3}^-\Omega_{\psi_i}(k) - (\omega^2q^{a_{i\sigma^2(j)}} + \omega q^{a_{i\sigma(j)}} + q^{a_{ij}})\gamma^2x_{j,s+2}^-\Omega_{\psi_i}(k+1)\\
         +  (\omega^2q^{a_{ij} + a_{i\sigma^2(j)}} + \omega q^{a_{ij} + a_{i\sigma(j)}} + q^{a_{i\sigma(j)} + a_{i\sigma^2(j)}})\gamma x_{j,s+1}^-\Omega_{\psi_i}(k+2) - q^{a_{ij} + a_{i\sigma(j)} + a_{i\sigma^2(j)}}x_{j,s}^-\Omega_{\psi_i}(k+3) \Bigl)
    \end{multline*}

\findem

\begin{lem}
    $\mc{N}_q^-$ is a left $\mc{K}_q$-module and $$ \mc{N}_q^- \cong \mc{K}_q/ \sum_{i\in I_\sigma}\sum_{k\in \Z} \mc{K}_q\Omega_{\psi_i}(k) $$
\end{lem}

\dem Thanks to Propositions \ref{Omega_x_commrealtions} and \ref{Omega_commrealtions}, $\mc{N}_q^-$ is a left $\mc{K}_q$-module. It is clear that the algebra $\mc{N}_q^-$ is a quotient of the Kashiwara algebra, and since $\Omega_{\psi_i}(k)$ annihilates 1 for all $k$ and $i\in I_\sigma$, there is a $\mc{K}_q$-epimorphism that sends 1 to 1 of the form $  \displaystyle \eta: \mc{K}_q/ \sum_{i\in I_\sigma}\sum_{k\in \Z} \mc{K}_q\Omega_{\psi_i}(k) \twoheadrightarrow \mc{N}_q^-$.

Consider now the subalgebra $C$ of $\mc{K}_q$ generated by $x^-_{i,p}
$ and $\gamma^{\pm 1/2}$ for all $p\in Z$. Then, there is a $\mc{K}_q$-epimorphism $\displaystyle \mu : C \twoheadrightarrow \mc{K}_q/ \sum_{i\in I_\sigma}\sum_{k\in \Z} \mc{K}_q\Omega_{\psi_i}(k)$. Clearly, $\eta \circ \mu$ is an epimorphism, and since $\mc{N}_q^-$ is generated by $x^-_{i,p}$ and $\gamma^{\pm 1/2}$ for all $p\in Z$ and $i\in I_\sigma$, and the relations (\ref{Kas_commxx_r2}) for $r=2$ and (\ref{Kas_commxx_r3}) for $r=3$ hold in it, there is an induced homomorphism $\nu : \mc{N}_q^- \to C$ such that $\nu\circ(\eta\circ\mu)$ is the identity. So, $\eta \circ\mu$ is an isomorphism, and hence $\eta$ is also an isomorphism.

\findem

\begin{prop}\label{Form_on_Nq}
    There exists a unique symmetric form $(\quad , \quad)$ defined on $\mc{N}_q^-$ satisfying $$ (x^-_{i,p}a,b) = (a,\Omega_{\psi_i}(-p)b), \quad \quad (1,1)=1 $$
    for $p\in \Z$, $i\in I_\sigma$ and $a,b\in \mc{N}_q^-$.
\end{prop}

\dem Using the anti-automorphism $\overline{\alpha}$, we can endow the set $M= \Hom(\mc{N}_q^-, \C(q^{1/2}))$ with the structure of a $\mc{K}_q$-module defined by 
$$ (x^-_{i,p}f)(a) = f(\Omega_{\psi_i}(-p)a), \quad (\Omega_{\psi_i}(k)f)(a) = f(x^-_{i,-k}a), \quad (\gamma^{\pm 1/2}f)(a) = f(\gamma^{\pm 1/2}a)  $$

where $p,m\in \Z$ and $f\in M$. Let $\beta_0\in M$ be the element such that $\beta_0(1) = 1$ and $\displaystyle \beta_0\left(\sum_{i\in I_\sigma}\sum_{p\in \Z} x_{i,p}^-\mc{N}_q^-\right)=0$. Then $\Omega_{\psi_i}(m)\beta_0 = 0$ for all $m\in \Z$, and we get an induced $\mc{K}_q$-module homomorphism $\displaystyle \mc{K}_q/ \sum_{i\in I_\sigma}\sum_{k\in \Z} \mc{K}_q\Omega_{\psi_i}(k) \to M$.\\

Composing this homomorphism with the inclusion $\mc{N}_q^- \to \mc{K}_q$ and the canonical projection $\displaystyle \mc{K}_q \to \mc{K}_q/ \sum_{i\in I_\sigma}\sum_{k\in \Z} \mc{K}_q\Omega_{\psi_i}(k)$, we get an induced $\mc{N}_q$-module homomorphism $\overline{\beta} : \mc{N}_q^- \to M$. Since $\overline{\beta}(1)=\beta_0$, we get that $\overline{\beta}$ is non-zero.\\

Define $(\quad,\quad):\mc{N}_q^- \times \mc{N}_q^- \to \C(q^{1/2})$ by 
$$ (a,b) = (\overline{\beta}(a))(b) $$

This form satisfies $(1,1) = (\overline{\beta}(1))(1) = \beta_0(1)=1$, $(\gamma^{\pm 1/2},1)=(\overline{\beta}(\gamma^{\pm 1/2}))(1) = 1$, and $(x^-_{i,p}a,b) = (\overline{\beta}(x^-_{i,p}a))(b) = (x^-_{i,p}\overline{\beta}(a))(b) =  \overline{\beta}(a)(\Omega_{\psi_i}(-p)b)= (a,\Omega_{\psi_i}(-p)b)$. Since $\mc{N}_q^-$ is generated by $x^-_p$ and $\gamma^{\pm 1/2}$, this form is the unique form satisfying these conditions. Moreover, since the form $(a,b)'=(b,a)$ satisfies the same conditions, the form is symmetric.

\findem

Finally, we show that $\mc{N}_q^-$ is a simple left $\mc{K}_q$-module.

\begin{lem}\label{lem:Omega-invariant}
Let $P\in\mc N_q^-$. Suppose that
\[
\Omega_{\psi_i}(k)P=0
\qquad
\text{for all }i\in I_\sigma,\ k\in\mathbb Z.
\]
Then $P\in\mathbb C1$.
\end{lem}

\dem
We first decompose $P$ with respect to the standard root-lattice
grading
\[
\mc N_q^-=
\bigoplus_{\beta\in Q_+}
(\mc N_q^-)_{-\beta}.
\]
Since the operators $\Omega_{\psi_i}(k)$ are homogeneous with
respect to this grading, it is enough to consider a nonzero
homogeneous element
\[
P\in(\mc N_q^-)_{-\beta}.
\]

Suppose that $\beta\neq0$. Let $d$ be the maximal number of
generators $x^-_{i,m}$ occurring in a monomial appearing in $P$.
Write
\[
P=P_d+P_{<d},
\]
where $P_d$ is the component consisting of monomials containing
exactly $d$ generators.

Fix a PBW ordering of the monomials and choose in $P_d$ a monomial
\[
M=x^-_{j,m}w
\]
which is maximal with respect to this ordering. Let $c\neq0$ be its
coefficient.

Choose $i\in I_\sigma$ and $s_0\in\{0,\ldots,r-1\}$ such that
\[
\sigma^{s_0}(i)=j.
\]
Applying $\Omega_{\psi_i}(-m)$ and using
\eqref{action_omega_on_x}, the contraction term gives
\[
c\,\omega^{s_0m}\gamma^{-m}w.
\]

The remaining contribution coming from the action on the first
factor is
\[
c\sum_{t\geq0}
g_{ij,q^{-1}}(t)
x^-_{j,0}
\Omega_{\psi_i}(-m-t)(w)\gamma^t.
\]
Since $\Omega_{\psi_i}(-m-t)$ acts on $w$, every nonzero term in
the sum contains one fewer generator in $w$, followed by the
additional generator $x^-_{j,0}$. Thus it has the same number of
generators as $w$.

We now compare these terms with the leading monomial $w$ using the
PBW ordering. By the choice of $M$ as the maximal monomial and by
the formulas (\ref{Kas_commxx_r2}) and (\ref{Kas_commxx_r3}), the term $\Omega_{\psi_i}(-m-t)w$ remains to be of degree $d-1$ and  strictly smaller than $w$. Likewise, the action on all monomials in $P_d$ which are strictly
smaller than $M$ produces only monomials strictly smaller than $w$.

Consequently, the coefficient of $w$ in
$\Omega_{\psi_i}(-m)P$ is
\[
c\,\omega^{s_0m}\gamma^{-m}\neq0.
\]
Hence
\[
\Omega_{\psi_i}(-m)P\neq0,
\]
contrary to the hypothesis. Therefore $\beta=0$, and hence $P\in\mathbb C1$.
\findem

\begin{thm}
The algebra $\mc{N}_q^-$ is a simple left $\mc{K}_q$-module.
\end{thm}

\begin{proof}
Let $0\neq M\subseteq\mc{N}_q^-$ be a $\mc{K}_q$-submodule.
Choose a nonzero element $P\in M$ of minimal degree with respect to
the grading
\[
\deg x^-_{i,p}=1,\qquad
\deg\gamma^{\pm1/2}=0.
\]
Since each operator $\Omega_{\psi_i}(k)$ lowers this degree by one,
minimality implies
\[
\Omega_{\psi_i}(k)P=0
\qquad
\text{for all }i\in I_\sigma,\ k\in\mathbb Z.
\]
By Lemma~\ref{lem:Omega-invariant}, $P$ is a nonzero scalar
multiple of $1$. Hence $1\in M$. Since $\mc{N}_q^-$ is generated by
the elements $x^-_{i,p}$ and $\gamma^{\pm1/2}$, it follows that
\[
M=\mc{N}_q^-.
\]
Thus $\mc{N}_q^-$ is simple as a left $\mc{K}_q$-module.
\end{proof}
\begin{cor}
    The form $(\quad,\quad)$ defined in Proposition \ref{Form_on_Nq} is non-degenerate.
\end{cor}

\dem The radical of the form is a $\mc{K}_q$-submodule of $\mc{N}_q^-$, and since $(1,1)=1$, it must be zero.
\findem

\bigskip
\begin{center}
ACKNOWLEDGMENT
\end{center}
\smallskip
 V. Futorny is partially supported by NSF of China (12350710787 and 12350710178)


\begin{thebibliography}{X}

\bibitem{AFM} Arias, J. C., Futorny, V. \& Misra, K. C. Crystal bases for reduced imaginary Verma modules of untwisted quantum affine algebras. {\em J. Algebra}. \textbf{655} pp. 3-28 (2024).

\bibitem{B01}Beck, J. Braid group action and quantum affine algebras. {\em Comm. Math. Phys.}. \textbf{165}, 555-568 (1994).

\bibitem{B02} Beck, J. Convex bases of PBW type for quantum affine algebras, {\em Comm. Math. Phys.} {\bf 165}, 193-199 (1994).

\bibitem{BK} Beck, J. \& Kac, V. Finite-dimensional representations of quantum affine algebras at roots of unity. {\em J. Amer. Math. Soc.}. \textbf{9}, 391-423 (1996).

\bibitem{BKMe} Benkart, G., Kang, S. \& Melville, D. Quantized enveloping algebras for Borcherds superalgebras. {\em Trans. Amer. Math. Soc.}. \textbf{350}, 3297-3319 (1998).

\bibitem{CP}Chari, V. \& Pressley, A. A guide to quantum groups. (Cambridge University Press, Cambridge,1995), Corrected reprint of the 1994 original.

\bibitem{CFM01}Cox, B., Futorny, V. \& Misra, K. Imaginary {V}erma modules and {K}ashiwara algebras for  {$U_q(\hat{\mathfrak{g}})$}. {\em J. Algebra}. \textbf{424} pp. 390-415 (2015).

\bibitem{D} Damiani, I. The R-matrix for (twisted) affine quantum algebras. {\em Representations And Quantizations (Shanghai, 1998)}. pp. 89-144 (2000).

\bibitem{D01}Drinfel'd, V. A new realization of Yangians and of quantum affine algebras. {\em Dokl. Akad. Nauk SSSR}. \textbf{296}, 13-17 (1987)

\bibitem{Fut}Futorny, V. Imaginary Verma modules for affine Lie algebras. {\em Canad. Math. Bull.}. \textbf{37}, 213-218 (1994)

\bibitem{FGM} Futorny, V., Grishkov A., \& Melville, D. Verma-type modules for quantum affine Lie algebras, Algebr. Represent. Theory {\bf 8}, 99--125 (2005).

\bibitem{J} Jing, N. On Drinfeld realization of quantum affine algebras. {\em The Monster And Lie Algebras (Columbus, OH, 1996)}. \textbf{7} pp. 195-206 (1998).

\bibitem{JM} Jing, N. \& Misra, K. Vertex operators for twisted quantum affine algebras. {\em Trans. Amer. Math. Soc.}. \textbf{351}, 1663-1690 (1999).

\bibitem{L01} Lusztig, G. Introduction to quantum groups. (Birkh\"auser/Springer, New York,2010), Reprint of the 1994 edition.

\bibitem{P} Papi, P. Convex orderings in affine root systems. {\em J. Algebra}. \textbf{172}, 613-623 (1995).



\end{thebibliography}
\end{document}